\documentclass[twoside]{amsart}

\usepackage{amsmath,amsfonts,amsthm,mathrsfs}
\usepackage{amssymb}

\usepackage[unicode,bookmarks]{hyperref}
\usepackage[usenames,dvipsnames]{xcolor}
\hypersetup{colorlinks=true,citecolor=NavyBlue,linkcolor=BrickRed,urlcolor=Orange}

\usepackage[alphabetic,initials]{amsrefs}

\usepackage{ifthen}

\usepackage{enumitem}

\usepackage{chngcntr}

\usepackage{tikz}
\usetikzlibrary{matrix,arrows}
\newlength{\myarrowsize} 

\pgfarrowsdeclare{cmto}{cmto}{
	\pgfsetdash{}{0pt} 
	\pgfsetbeveljoin 
	\pgfsetroundcap 
	\setlength{\myarrowsize}{0.6pt}
	\addtolength{\myarrowsize}{.5\pgflinewidth}
	\pgfarrowsleftextend{-4\myarrowsize-.5\pgflinewidth} 
	\pgfarrowsrightextend{.8\pgflinewidth}
}{
	\setlength{\myarrowsize}{0.6pt} 
  	\addtolength{\myarrowsize}{.5\pgflinewidth}  
	\pgfsetlinewidth{0.5\pgflinewidth}
	\pgfsetroundjoin
	\pgfpathmoveto{\pgfpoint{1.5\pgflinewidth}{0}}
	\pgfpatharc{-109}{-170}{4\myarrowsize}
	\pgfpatharc{10}{189}{0.58\pgflinewidth and 0.2\pgflinewidth}
	\pgfpatharc{-170}{-115}{4\myarrowsize+\pgflinewidth}
	\pgfpathclose
	\pgfusepathqfillstroke
	\pgfpathmoveto{\pgfpoint{1.5\pgflinewidth}{0}}
	\pgfpatharc{109}{170}{4\myarrowsize}
	\pgfpatharc{-10}{-189}{0.58\pgflinewidth and 0.2\pgflinewidth}
	\pgfpatharc{170}{115}{4\myarrowsize+\pgflinewidth}
	\pgfpathclose
	\pgfusepathqfillstroke
	\pgfsetlinewidth{2\pgflinewidth}
}

\pgfarrowsdeclare{cmonto}{cmonto}{
	\pgfsetdash{}{0pt} 
	\pgfsetbeveljoin 
	\pgfsetroundcap 
	\setlength{\myarrowsize}{0.6pt}
	\addtolength{\myarrowsize}{.5\pgflinewidth}
	\pgfarrowsleftextend{-4\myarrowsize-.5\pgflinewidth} 
	\pgfarrowsrightextend{.8\pgflinewidth}
}{
	\setlength{\myarrowsize}{0.6pt} 
  	\addtolength{\myarrowsize}{.5\pgflinewidth}  
	\pgfsetlinewidth{0.5\pgflinewidth}
	\pgfsetroundjoin
	\pgfpathmoveto{\pgfpoint{1.5\pgflinewidth}{0}}
	\pgfpatharc{-109}{-170}{4\myarrowsize}
	\pgfpatharc{10}{189}{0.58\pgflinewidth and 0.2\pgflinewidth}
	\pgfpatharc{-170}{-115}{4\myarrowsize+\pgflinewidth}
	\pgfpathclose
	\pgfusepathqfillstroke
	\pgfpathmoveto{\pgfpoint{1.5\pgflinewidth}{0}}
	\pgfpatharc{109}{170}{4\myarrowsize}
	\pgfpatharc{-10}{-189}{0.58\pgflinewidth and 0.2\pgflinewidth}
	\pgfpatharc{170}{115}{4\myarrowsize+\pgflinewidth}
	\pgfpathclose
	\pgfusepathqfillstroke
	\pgfpathmoveto{\pgfpoint{1.5\pgflinewidth-0.3em}{0}}
	\pgfpatharc{-109}{-170}{4\myarrowsize}
	\pgfpatharc{10}{189}{0.58\pgflinewidth and 0.2\pgflinewidth}
	\pgfpatharc{-170}{-115}{4\myarrowsize+\pgflinewidth}
	\pgfpathclose
	\pgfusepathqfillstroke
	\pgfpathmoveto{\pgfpoint{1.5\pgflinewidth-0.3em}{0}}
	\pgfpatharc{109}{170}{4\myarrowsize}
	\pgfpatharc{-10}{-189}{0.58\pgflinewidth and 0.2\pgflinewidth}
	\pgfpatharc{170}{115}{4\myarrowsize+\pgflinewidth}
	\pgfpathclose
	\pgfusepathqfillstroke
	\pgfsetlinewidth{2\pgflinewidth}
}

\pgfarrowsdeclare{cmhook}{cmhook}{
	\pgfsetdash{}{0pt} 
	\pgfsetbeveljoin 
	\pgfsetroundcap 
	\setlength{\myarrowsize}{0.6pt}
	\addtolength{\myarrowsize}{.5\pgflinewidth}
	\pgfarrowsleftextend{-4\myarrowsize-.5\pgflinewidth} 
	\pgfarrowsrightextend{.8\pgflinewidth}
}{
	\setlength{\myarrowsize}{0.6pt} 
  	\addtolength{\myarrowsize}{.5\pgflinewidth}  
 	\pgfsetdash{}{0pt}
	\pgfsetroundcap
	\pgfpathmoveto{\pgfqpoint{0pt}{-4.667\pgflinewidth}}
	\pgfpathcurveto
    {\pgfqpoint{4\pgflinewidth}{-4.667\pgflinewidth}}
    {\pgfqpoint{4\pgflinewidth}{0pt}}
    {\pgfpointorigin}
	\pgfusepathqstroke
}

\newenvironment{diagram}[2]{%
\begin{equation}%
\begin{tikzpicture}[>=cmto,baseline=(current bounding box.center),%
	to/.style={-cmto,font=\scriptsize,cap=round},%
	into/.style={cmhook->,font=\scriptsize,cap=round},%
	onto/.style={-cmonto,font=\scriptsize,cap=round},%
	math/.style={matrix of math nodes, row sep=#2, column sep=#1,%
		text height=1.5ex, text depth=0.25ex}]%
}{%
\end{tikzpicture}%
\end{equation}%
\ignorespacesafterend%
}
\newenvironment{diagram*}[2]{%
\[%
\begin{tikzpicture}[>=cmto,baseline=(current bounding box.center),%
	to/.style={->,font=\scriptsize,cap=round},%
	into/.style={cmhook->,font=\scriptsize,cap=round},%
	onto/.style={-cmonto,font=\scriptsize,cap=round},%
	math/.style={matrix of math nodes, row sep=#2, column sep=#1,%
		text height=1.5ex, text depth=0.25ex}]%
}{%
\end{tikzpicture}%
\]%
\ignorespacesafterend%
}

\newcommand{\MHM}{\operatorname{MHM}}

\newcommand{\Dmod}{\mathscr{D}}
\newcommand{\Mmod}{\mathcal{M}}
\newcommand{\Nmod}{\mathcal{N}}

\newcommand{\derR}{\mathbf{R}}
\newcommand{\derL}{\mathbf{L}}

\newcommand{\decal}[1]{\lbrack #1 \rbrack}

\newcommand{\ltriangle}[4][]%
{\begin{diagram}[#1]%
	{#2} &\rTo& {#3} &\rTo& {#4} &\rTo& {#2 \decal{1}}%
\end{diagram}}

\newcommand{\shH}{\mathcal{H}}

\newcommand{\eps}{\varepsilon}

\newcommand{\tensor}{\otimes}

\newcommand{\NN}{\mathbb{N}}
\newcommand{\ZZ}{\mathbb{Z}}
\newcommand{\QQ}{\mathbb{Q}}

\newcommand{\CC}{\mathbb{C}}

\DeclareMathOperator{\DR}{DR}

\newcommand{\shf}[1]{\mathscr{#1}}
\newcommand{\OX}{\shf{O}_X}

\def\overbar#1#2#3{{%
	\setbox0=\hbox{\displaystyle{#1}}%
	\dimen0=\wd0
	\advance\dimen0 by -#2 
	\vbox {\nointerlineskip \moveright #3 \vbox{\hrule height 0.3pt width \dimen0}%
		\nointerlineskip \vskip 1.5pt \box0}%
}}

\newcommand{\shF}{\shf{F}}

\newcommand{\shO}{\shf{O}}

\makeatletter
\let\@@seccntformat\@seccntformat
\renewcommand*{\@seccntformat}[1]{%
  \expandafter\ifx\csname @seccntformat@#1\endcsname\relax
    \expandafter\@@seccntformat
  \else
    \expandafter
      \csname @seccntformat@#1\expandafter\endcsname
  \fi
    {#1}%
}
\newcommand*{\@seccntformat@subsection}[1]{%
  \textbf{\csname the#1\endcsname.}
}
\makeatother

\makeatletter
\let\@paragraph\paragraph
\renewcommand*{\paragraph}[1]{%
	\vspace{0.3\baselineskip}%
	\@paragraph{\textit{#1}}%
}
\makeatother

\counterwithin{equation}{subsection}
\counterwithout{subsection}{section}
\counterwithin{figure}{subsection}

\newtheorem{theorem}[equation]{Theorem}
\newtheorem*{theorem*}{Theorem}
\newtheorem{lemma}[equation]{Lemma}
\newtheorem*{lemma*}{Lemma}
\newtheorem{corollary}[equation]{Corollary}
\newtheorem{proposition}[equation]{Proposition}
\newtheorem*{proposition*}{Proposition}

\theoremstyle{definition}
\newtheorem{definition}[equation]{Definition}
\newtheorem*{definition*}{Definition}
\theoremstyle{remark}
\newtheorem{remark}[equation]{Remark}

\newtheorem{question}[equation]{Question}

\newtheorem{example}[equation]{Example}
\newtheorem*{example*}{Example}

\theoremstyle{plain}

\newcommand{\theoremref}[1]{\hyperref[#1]{Theorem~\ref*{#1}}}
\newcommand{\lemmaref}[1]{\hyperref[#1]{Lemma~\ref*{#1}}}
\newcommand{\propositionref}[1]{\hyperref[#1]{Proposition~\ref*{#1}}}
\newcommand{\conjectureref}[1]{\hyperref[#1]{Conjecture~\ref*{#1}}}
\newcommand{\corollaryref}[1]{\hyperref[#1]{Corollary~\ref*{#1}}}

\makeatletter
\let\old@caption\caption
\renewcommand*{\caption}[1]{%
	\setcounter{figure}{\value{equation}}%
	\stepcounter{equation}%
	\old@caption{#1}\relax%
}
\makeatother

\newcounter{thmA}
\newtheorem{theorem-intro}[thmA]{Theorem}

\newcounter{intro}

\newtheorem{intro-conjecture}[intro]{Conjecture}
\newtheorem{intro-corollary}[intro]{Corollary}
\newtheorem{intro-theorem}[intro]{Theorem}

\newcommand{\OA}{\mathscr{O}_A}

\newcommand{\OmA}[1]{\Omega_A^{#1}}

\newcommand{\PP}{\mathbf P}
\newcommand{\V}{\mathbf V}
\newcommand{\cV}{\mathcal{V}}
\newcommand{\R}{\mathbf R}
\newcommand{\D}{\mathbf D}

\newcommand{\I}{\mathcal{I}}

\def\lra{\longrightarrow}
\DeclareMathOperator{\BB}{{\bf B}}
\DeclareMathOperator{\B+}{\BB_{+}}

\DeclareMathOperator{\Gr}{Gr}

\newcommand{\newpar}[1]{\subsection{\texorpdfstring{}{}}}
\newcommand{\parref}[1]{\hyperref[#1]{\S\ref*{#1}}}

\begin{document}

\title{Kodaira-Saito vanishing and applications}

\author[M.~Popa]{Mihnea Popa}
\address{Department of Mathematics, Northwestern University,
2033 Sheridan Road, Evanston, IL 60208, USA} 
\email{\tt mpopa@math.northwestern.edu}



\thanks{During the preparation of this paper I was supported by the NSF grant DMS-1405516.}

\begin{abstract}
The first part of the paper contains a detailed proof of M. Saito's generalization of the Kodaira vanishing theorem,
following the original argument and with ample background. The second part contains some recent applications, 
and a Kawamata-Viehweg-type statement in the setting of mixed Hodge modules.

\end{abstract}

\maketitle


\subsection{Introduction}\label{intro}
This article was originally the outcome of a lecture delivered at the Clay workshop on mixed Hodge modules, held at Oxford University in August 2013. The main goal was to explain in detail the proof of Morihiko Saito's extension of the Kodaira-Nakano vanishing theorem to mixed Hodge modules, discuss various special cases, and give a guide to recent applications. This is done in the first and main part of the paper, Sections \ref{classical}--\ref{particular_cases}, which also includes ample background. Since then I have also included some new applications. One is a proof of weak positivity for the lowest graded piece of a Hodge module obtained jointly with C. Schnell (which also appears in \cite{Schnell2}). Another is a Hodge module version of the Kawamata-Viehweg vanishing theorem, likely not in its final form.\footnote{Added during revision: in the meanwhile,
in the case of Cartier divisors a stronger Kawamata-Viehweg-type vanishing theorem was indeed proved by  Suh \cite{Suh} and Wu \cite{Wu}.}

M. Saito's vanishing theorem is stated and proved as Theorem \ref{saito_vanishing} below. It was obtained in \cite{Saito-MHP}*{\S2.g}; the proof provided here is a detailed account of Saito's original argument, which in turn is a generalization of Ramanujam's topological approach to vanishing.  C. Schnell \cite{Schnell3} has recently found a different proof of the theorem, this time extending the Esnault-Viehweg approach to vanishing via the degeneration of the Hodge-to-de Rham spectral sequence on cyclic covers.

In order to make the underlying approach of Saito clear, I will first recall the proof of the Kodaira-Nakano vanishing theorem based on the weak Lefschetz theorem, the Hodge decomposition, and cyclic covering constructions. In the proof of Theorem \ref{saito_vanishing}, the corresponding roles will be played by the Artin-Grothendieck vanishing theorem for constructible sheaves and by M. Saito's generalization of the standard results of Hodge theory to the setting of mixed Hodge modules. There are however significant new difficulties that are resolved with the use of the interaction between the Hodge filtration and the Kashiwara-Malgrange $V$-filtration established in \cite{Saito-MHP}, recalled in the preliminaries; the background discussion will survey this and other facts about filtered $\Dmod$-modules in Hodge theory, with references for all the statements needed in the paper. 

Many of the standard vanishing theorems involving ample line bundles are special cases of Saito vanishing. This will be reviewed in Section \ref{particular_cases}, where I will also mention its use to generic vanishing theory. When passing to big and nef line $\QQ$-divisors however, the situation is more complicated. In Section \ref{kawamata_viehweg} I prove a first version of Kawamata-Viehweg for mixed Hodge modules -- 
roughly speaking, it assumes that the Hodge module is a variation of mixed Hodge structure over the augmented base locus of a nef and big line bundle. Another application, provided in Section \ref{weak_positivity}, is a proof together with Schnell of an extension of a weak positivity theorem of Viehweg to the lowest graded piece of the Hodge filtration on a Hodge $\Dmod$-module. Arguing along the lines of Koll\'ar's approach to weak positivity provides a very quick argument, once Kodaira vanishing and adjunction have been extended to setting of mixed Hodge modules.

As a good part of the paper is expository, my main goal is to make these very useful statements and techniques more accessible to algebraic geometers; the 
viewpoint is that of cohomological methods in birational geometry. The reader interested in a more general overview of the theory of mixed Hodge modules is encouraged to consult the recent \cite{Schnell-MHM}, besides of course the original \cite{Saito-MHP} and \cite{Saito-MHM}. 

\noindent
{\bf Acknowledgements.}
I am very grateful to Christian Schnell, from whom I learned a lot about Hodge modules, and who made numerous useful comments on this paper. I would also like to thank Nero Budur, Mircea Musta\c t\u a, Claude Sabbah and Morihiko Saito for answering my questions, and the organizers of the Oxford Clay workshop on mixed Hodge modules (all among the above) for putting together such a valuable event.

\subsection{The topological/Hodge theoretic approach to Kodaira vanishing}\label{classical}
In this section I will recall the approach to the Kodaira vanishing theorem based on topological and Hodge theoretic methods, which also gives the more general Nakano vanishing.
It was first observed by Ramanujam that one can use such methods, Kodaira's original proof being
 of a differential geometric nature. 
I will follow the treatment in \cite{Lazarsfeld} \S4.2; this is intended to be an introduction to the strategy 
used by Saito in order to prove the more general result for Hodge modules. 

\begin{theorem}[{\bf Kodaira-Nakano Vanishing Theorem}]\label{nakano}
Let $X$ be a smooth complex projective variety, and $L$ an ample line bundle on $X$. Then 
$$H^q (X, \Omega_X^p \otimes L) = 0 \,\,\,\, {\rm for~} p + q >n,$$
or equivalently
$$H^q (X, \Omega_X^p \otimes L^{-1}) = 0 \,\,\,\, {\rm for~} p + q  < n.$$
\end{theorem}

Before proving the theorem, let's review some useful technical tools. 
First, recall the following well-known cyclic covering construction, needed in order to 
 ``take $m$-th roots" of divisors $D \in |mL|$, with $L$ some line bundle. For a proof of this and other 
 covering constructions see \cite{Lazarsfeld} \S4.1.B.

\begin{proposition}\label{cyclic}
Let $X$ be a variety over an algebraically closed field $k$, and let $L$ be a line bundle on $X$. Let 
$0 \neq s \in H^0 (X, L^{\otimes m})$ 
for some $m \ge 1$, with $D = Z (s) \in |mL|$. Then there exists a finite flat morphism $f: Y \rightarrow X$, where $Y$ is a 
scheme over $k$ such that if $L^\prime = f^* L$, there is a section 
$$s^\prime \in H^0 (Y, L^\prime) {\rm~satisfying~} {(s^\prime)}^m = f^* s.$$
Moreover:

\noindent
$\bullet$ if $X$ and $D$ are smooth, then so are $Y$ and $D^\prime = Z(s^\prime)$.

\noindent
$\bullet$ the divisor $D^\prime$ maps isomorphically onto $D$. 

\noindent
$\bullet$ there is a canonical isomorphism $f_* \shO_Y \simeq \shO_X \oplus L^{-1} \oplus \cdots\oplus L^{- (m-1)}$.
\end{proposition}

Furthermore, recall that if $X$ is a smooth variety, and $D$ is a smooth effective divisor on $X$, 
then the sheaf of $1$-forms on $X$ with log-poles along $D$ is
$$\Omega_X^1 ({\rm log}~D) = \Omega_X^1 < \frac{df}{f}> , \,\, f {\rm ~local ~equation~for~} D.$$
Concretely, if $z_1, \ldots, z_n$ are local coordinates on $X$, chosen such that $D = (z_n = 0)$, then 
$\Omega_X^1 ({\rm log}~D)$ is locally generated by $dz_1, \ldots, dz_{n-1}, \frac{dz_n}{z_n}$. This is a free 
system of generators, so $\Omega_X^1 ({\rm log}~D)$ is locally free of rank $n$. For any integer $p$, we define
$$\Omega_X^p ({\rm log}~D) := \bigwedge^p \bigl( \Omega_X^1 ({\rm log}~D) \bigr).$$
Using local calculations and the residue map, it is standard to verify the following statements (see \cite{EV}*{\S2}
or \cite{Lazarsfeld}*{Lemma 4.2.4}):

\begin{lemma}\label{log_sequences}
There are short exact sequences:

\noindent
(i) \,\,\,\, $0 \lra \Omega_X^p \lra \Omega_X^p ({\rm log}~ D) \lra \Omega_D^{p-1} \lra 0$.

\noindent
(ii) \,\,\,\, $0 \lra \Omega_X^p ({\rm log}~D) (- D) \lra \Omega_X^p  \lra \Omega_D^p \lra 0$.
\end{lemma}

\begin{lemma}\label{log_cover}
Let $f: Y \rightarrow X$ be the $m$-fold cyclic cover branched along $D$, as in Proposition 
\ref{cyclic}. Let $D^\prime$ be the divisor in $Y$ such that $f^* D = mD^\prime$, mapping isomorphically onto $D$. Then 
$$f^*  \Omega_X^p ({\rm log}~D) \simeq  \Omega_Y^p ({\rm log}~D^\prime).$$
\end{lemma}

\begin{proof}[Sketch of proof of Theorem \ref{nakano}]
By Serre duality it suffices to show the second part of the statement. 
For $m \gg 0$, let $D \in |mL|$ be a smooth divisor.
One can assume by induction on $n = {\rm dim}~X$ that we already know Kodaira-Nakano vanishing 
on $D$, so that 
$$H^q (D, \Omega_D^{p-1} \otimes L_{|D}^{-1} ) = 0 \,\,\,\, {\rm for~} p + q < n.$$
Using this and passing to cohomology in the sequence in Lemma \ref{log_sequences}(i), it suffices then to prove that
$$H^q (X, \Omega_X^p ({\rm log}~D) \otimes L^{-1} ) = 0 \,\,\,\, {\rm for~} p + q < n.$$

Let now $f: Y \rightarrow X$ be the $m$-fold cyclic cover branched along $D$ as 
in Proposition \ref{cyclic}, with $f^* D = mD^\prime$ and $L^\prime = \shO_Y (D^\prime)$. 
 Proposition  \ref{cyclic} says that $Y$ and $D^\prime$ can be chosen to be smooth; also, 
 $D^\prime$ is obviously ample. 
Since $f$ is a finite cover, using Lemma \ref{log_cover} what we want is equivalent to showing that
$$H^q (Y, \Omega_Y^p ({\rm log}~D^\prime) \otimes \shO_Y (- D^\prime)) = 0 \,\,\,\, {\rm for~} p + q < n.$$
One can now appeal to the exact sequence in Lemma \ref{log_sequences}(ii). Using this, our desired 
statement is equivalent to the fact that the restriction maps
$$r_{p,q} : H^q (Y, \Omega_Y^p) \longrightarrow H^q (D^\prime, \Omega_{D^\prime}^p)$$
are isomorphisms for $p+q \le n-2$, and injective for $p+q = n-1$. But this follows immediately 
from the weak Lefschetz theorem, as the restriction maps
$$H^i (Y, \CC) \longrightarrow H^i (D^\prime, \CC)$$
are morphisms of Hodge structures.
\end{proof}

Saito's generalization of Theorem \ref{nakano} is stated and proved in Section \ref{main}, 
while important special cases are explained in Section \ref{particular_cases}. 
Before being able to do this we need a lengthy review of background material. 
The reader may already visit those sections however, for a first encounter with the main topic.

\subsection{Filtered $\Dmod$-modules and de Rham complexes}\label{filtered}

In this section I will recall some filtered $\Dmod$-module terminology and facts used in the paper. Excellent 
introductions to the subject are for instance the book by Hotta-Takeuchi-Tanisaki \cite{HTT} and the lecture notes of Maisonobe-Sabbah \cite{Maisonobe_Sabbah}.
In what follows the standard language is that of right $\Dmod$-modules; as emphasized in \cite{Saito-MHP}, this is often more appropriate in the theory of mixed Hodge modules, for instance due to the fact that it is the natural setting for considering direct image or duality functors.  Occasionally however left $\Dmod$-modules 
will be necessary, in which case I will state explicitly that we are considering that setting and are performing the left-right transformation described below. 

\noindent
{\bf Definitions.}
Let $X$ be a  smooth complex variety. A filtered right $\Dmod$-module on $X$ is a $\Dmod_X$-module with an increasing filtration $F = F_\bullet \Mmod$ by coherent $\shO_X$-modules, bounded from below and satisfying
$$F_l \Mmod \cdot F_k \Dmod_X \subseteq F_{k+l} \Mmod \,\,\,\,{\rm for~all~} k, l \in \ZZ.$$
In addition, the filtration is \emph{good} if the inclusions above are equalities for $k \gg 0$. This condition is equivalent to the fact that the total 
associated graded object
$$\Gr^F_{\bullet} \Mmod = \bigoplus_k \Gr_k^F \Mmod = \bigoplus_k F_k \Mmod / F_{k-1} \Mmod$$
is finitely generated over $\Gr_{\bullet}^F \Dmod_X \simeq {\rm Sym}~T_X$, i.e. induces a coherent sheaf on the cotangent bundle $T^*X$.
Assuming that such a good filtration exists (in which case $\Mmod$ is also called coherent), the closed subset 
$${\rm Char}( \Mmod) : = {\rm Supp}~ \Gr^F_{\bullet} \Mmod  \subseteq T^* X$$
is called the characteristic variety of $X$. A well-known result of Bernstein says that $\dim {\rm Ch}(\Mmod) \ge \dim X$, and $\Mmod$ is 
called \emph{holonomic} if this is actually an equality. The $\Dmod$-modules we consider later will only be of this kind.

\noindent
{\bf Left-right rule.}
The canonical bundle $\omega_X$ is naturally endowed with a right $\Dmod_X$-module structure. Concretely, if $z_1, \ldots, z_n$ are local coordinates on $X$, for any $f \in \shO_X$ and any $P \in \Dmod_X$, the action is
$$(f \cdot dz_1 \wedge \cdots \wedge dz_n) \cdot P = {}^tP (f) \cdot dz_1 \wedge \cdots \wedge dz_n.$$ 
Here, if $P = \sum_{\alpha} g_\alpha \partial^\alpha$, then ${}^t P =  \sum_{\alpha} (-\partial)^\alpha g_\alpha$ is its formal adjoint.

Using this structure, as one often needs to switch between the two, let's recall 
the one-to-one correspondence between left and right $\Dmod_X$-modules given by 
$$\Nmod \mapsto \Mmod = \Nmod \otimes{_{\shO_X}} \omega_X \,\,\,\, {\rm and} \,\,\,\, \Mmod \mapsto \Nmod = \mathcal{H}om_{\shO_X} (\omega_X, \Mmod).$$
In terms of filtrations, the left-right rule is
$$F_p \Nmod = F_{p- n} \Mmod \otimes_{\shO_X} \omega_X^{-1}.$$

\noindent
{\bf de Rham complex.}
While we will consider right $\Dmod_X$-modules when talking about Hodge modules, one naturally associates
the de Rham complex to the corresponding left $\Dmod_X$-module $\Nmod$:
\[
\DR_X(\Nmod) = \Bigl\lbrack
		\Nmod \to \Omega_X^1 \tensor \Nmod \to \dotsb \to \Omega_X^n \tensor \Nmod
	\Bigr\rbrack,
\]
which is a $\CC$-linear complex placed in degrees $0, \dotsc, n$, with maps induced by the corresponding integrable connection
$\nabla: \Nmod \rightarrow \Nmod \otimes \Omega_X^1$. It turns out that the natural de Rham complex to consider for 
the right $\Dmod$-module $\Mmod$ (sometimes called a Spencer complex; see \cite{Maisonobe_Sabbah}*{1.4.2}) satisfies
$$\DR_X (\Mmod) \simeq \DR_X (\Nmod) [n].$$
 
By definition the filtration $F_{\bullet} \Mmod$ is
compatible with the $\Dmod_X$-module structure on $\Mmod$ and therefore, using the left-right rule above, this induces a 
filtration on the de Rham complex of $\Mmod$ by the formula
\[
	F_k \DR_X(\Mmod) = \Bigl\lbrack
		\bigwedge^n T_X \otimes F_{k-n} \Mmod \to \bigwedge^{n-1} T_X \tensor F_{k+1 -n} \Mmod \to \dotsb 
			\to  F_{k} \Mmod
	\Bigr\rbrack [n].
\]
The associated graded complexes for the filtration above are 
\[
	\Gr_k^F \DR_X(\Mmod) = \Big\lbrack
		\bigwedge^n T_X \otimes  \Gr_{k-n}^F \Mmod \to  \bigwedge^{n-1} T_X \otimes  \Gr_{k+1- n}^F \Mmod \to \dotsb \to
			\Gr_{k}^F \Mmod
	\Big\rbrack [n],
\]
which are now complexes of coherent $\OX$-modules in degrees $-n, \dotsc, 0$, 
and provide objects in $\D^b (X)$, the bounded derived category of coherent sheaves on $X$.

We will be particularly interested in the lowest non-zero graded piece of a filtered $\Dmod$-module. 
For one such right $\Dmod_X$-module $(\Mmod, F)$ define
\begin{equation}\label{lowest_def}
p (\Mmod) : = {\rm min}~\{p ~|~ F_p \Mmod \neq 0\}\,\,\,\,{\rm and} \,\,\,\, S (\Mmod) := F_{p(\Mmod)} \Mmod.
\end{equation}
For the associated left $\Dmod_X$-module we then have
$$p(\Nmod) = p (\Mmod) + n \,\,\,\, {\rm and} \,\,\,\,S(\Nmod) = S(\Mmod) \otimes \omega_X^{-1}. $$

\noindent
{\bf Pushforward.}
Let $f\colon X \rightarrow Y$ be a morphism of smooth complex varieties. We consider the associated transfer module
$$\Dmod_{X\to Y} : = \shO_X \otimes_{f^{-1} \shO_Y} f^{-1} \Dmod_Y.$$
It has the structure of a $(\Dmod_X, f^{-1} \Dmod_Y)$-bimodule, and it has a filtration given by 
$f^* F_k \Dmod_Y$. For a right $\Dmod_X$-module $\Mmod$, one can define a naive pushforward  as
$$f_* \Mmod : = f_* \big(\Mmod \otimes_{\Dmod_X} \Dmod_{X\to Y} \big),$$
where on the right hand side $f_*$ is the usual sheaf-theoretic direct image. However, the appropriate pushforward is 
in fact at the level of derived categories, namely
$$f_+ : {\bf D} (\Dmod_X) \longrightarrow {\bf D} (\Dmod_Y), \,\,\,\,\, \Mmod^{\bullet} \mapsto 
\derR f_* \big(\Mmod^{\bullet} \overset{\derL}{\otimes}_{\Dmod_X} \Dmod_{X\to Y} \big).$$
This is due to the left exactness of $f_*$ versus the right exactness of $\otimes$. See \cite[\S1.5]{HTT} for more details; in \emph{loc. cit.}
this last functor is denoted by $\int_f$.

Given a proper morphism of smooth varieties $f: X \rightarrow Y$, Saito has also constructed in \cite[\S2.3]{Saito-MHP} 
a filtered direct image functor 
$$f_+ : \D^b \big({\rm FM}(\Dmod_X)\big) \rightarrow \D^b \big({\rm FM}(\Dmod_Y)\big).$$
Here the two categories are the bounded derived categories of filtered $\Dmod$-modules on $X$ and $Y$ respectively.
Without filtration, it is precisely the functor above. The filtration requires more work; I will include  a few details below for the special 
$\Dmod$-modules that we consider in this paper.

\noindent
{\bf Strictness.}
A special property that is crucial in the theory of filtered $\Dmod$-modules underlying Hodge modules is the strictness of the filtration.  
Let 
$$f: (\Mmod, F) \rightarrow  (\Nmod, F)$$
be a morphism of filtered $\Dmod_X$-modules, i.e. such that 
$f(F_k \Mmod) \subseteq F_k \Nmod$ for all $k$. Then $f$ is called \emph{strict} if
$$f( F_k \Mmod)  = F_k \Nmod \cap f(\Mmod) \,\,\,\,\,\, {\rm for ~all~}k.$$ 
Similarly, a complex of filtered $\Dmod_X$-modules  $(\Mmod^\bullet, F_{\bullet} \Mmod^\bullet)$ is called strict if all of its differentials are strict. 
It can be easily checked that an equivalent interpretation is the following: the complex is strict if and only if, for every $i, k \in \ZZ$, we have that the induced  morphism 
$$\mathcal{H}^i (F_k \Mmod^\bullet) \longrightarrow \mathcal{H}^i \Mmod^\bullet$$
is injective. It is only in this case that the cohomologies of $\Mmod^\bullet$  can also be seen as filtered $\Dmod_X$-modules.

Via a standard argument, the notion of strictness makes sense more generally for objects in the derived category ${\bf D}^b \big({\rm FM}(\Dmod_X)\big)$ of filtered  $\Dmod_X$-modules. In the next sections, a crucial property of the filtered $\Dmod$-modules we consider is the following. If $f: X \rightarrow Y$ is a proper morphism of smooth varieties, and $(\Mmod, F)$ is one such filtered right $\Dmod_X$-module, then $f_+ (\Mmod, F)$ is strict as an object in $\D^b \big({\rm FM}(\Dmod_Y)\big)$;  here $f_+$ is the filtered direct image functor mentioned 
above. Given the previous discussion, this means that 
$$H^i \big( F_k f_+ (\Mmod, F)\big) \rightarrow H^i f_+ (\Mmod, F)$$
is injective  for all integers $i$ and $k$.
Finally, Saito's definition of the filtration on the direct image implies that this is equivalent to
the injectivity of the mapping
$$R^i f_* \big(F_k (\Mmod \overset{\derL}{\otimes}_{\Dmod_X} \Dmod_{X\to Y}) \big) \rightarrow 
R^i f_* (\Mmod \overset{\derL}{\otimes}_{\Dmod_X} \Dmod_{X\to Y}).$$
Up to a choice of representatives, the image is the filtration $F_k H^i f_+ (\Mmod, F)$. Thus in the strict case, one has a reasonably 
good grasp of the filtration on direct images, and the cohomologies of direct images are themselves filtered $\Dmod$-modules. Even more is true
in case $(\Mmod, F)$ underlies a Hodge module, as we will see in the next section.

\subsection{Hodge modules and variations of Hodge structure}\label{VHS}
Starting with this section, and up to \S\ref{V-filtration}, I will recall the objects that are the main focus of the paper. In the next section I will give several important examples. The main two references for the theory of Hodge modules are Morihiko Saito's papers \cite{Saito-MHP} in the pure case, and \cite{Saito-MHM} in the mixed case. A quite gentle but comprehensive overview of the theory was recently provided by Schnell \cite{Schnell-MHM}. Here I will give a brief review of the information needed for understanding the 
statement and proof of Saito's vanishing theorem;  the reader is encouraged to 
consult the references above for further information.

Let us first recall the notion of a variation of Hodge structure, which is the ``smooth" version of a Hodge module.
notion. If $X$ is a smooth complex variety, a variation of $\QQ$-Hodge structure of weight $\ell$ on $X$ is the data 
$$\V = (\mathcal{V}, F^\bullet, \V_{\QQ})$$
where: 

\noindent
$\bullet$~$\V_{\QQ}$ is a $\QQ$-local system on $X$. 

\noindent
$\bullet$~ ~$\mathcal{V} = \V_{\QQ}\otimes_{\QQ} \shO_X$ is a vector bundle with flat connection $\nabla$, 
endowed with a decreasing filtration with subbundles $F^p = F^p \mathcal{V}$ satisfying the following two properties:

\noindent
$\bullet$~for all $x \in X$, the data $\V_x = (\mathcal{V}_x, F^\bullet_x, \V_{\QQ, x})$ is a Hodge structure of weight $\ell$.

\noindent
$\bullet$~Griffiths transversality: for each $p$, $\nabla$ induces a morphism
$$\nabla: F^p \longrightarrow F^{p-1} \otimes \Omega_X^1.$$

Considering the Tate twist $\QQ(-\ell) = (2\pi i)^{-\ell} \QQ$, a polarization on $\V$ is a morphism 
$$Q: \V_{\QQ} \otimes \V_{\QQ} \longrightarrow \QQ(-\ell)$$
inducing a polarization of the Hodge structure $\V_x$ for each $x \in X$. We say that $\V$ is polarizable if one 
such polarization exists.

In order to generalize this notion, consider now $X$ to be a smooth complex algebraic variety of dimension $n$, 
and let $Z$ be an irreducible closed subset. 
Let $\V = (\mathcal{V}, F^\bullet, \V_{\QQ})$ be a polarizable variation of $\QQ$-Hodge structure of weight $\ell$ on an open set $U$ in the smooth locus of $Z$. Following \cite{Saito-MHP}, one can change terminology and call it a smooth pure Hodge module of weight $\dim Z + \ell$ on U, whose main constituents are:

\noindent
(i)~The right $\Dmod$-module $\Mmod = \mathcal{V} \otimes \omega_U$ with filtration
$F_p \Mmod =  F^{-p-n} \mathcal{V} \otimes \omega_U$.

\noindent
(ii)~The $\QQ$-perverse sheaf $P = \V_{\QQ} [n]$. 

According to Saito's theory, this extends uniquely to a \emph{pure polarizable Hodge module} $M$ of weight $\dim Z + \ell$ on $X$, whose support is $Z$. This has an underlying perverse sheaf, which is the intersection complex ${\rm IC}_Z( \V_{\QQ})=  {}^pj_{!*}\V_{\QQ}$ associated to the given local system. For this reason one sometimes uses the notation $M : = j_{!*}\V$. 
It also has an underlying $\Dmod$-module, namely the minimal extension of $\Mmod$, corresponding to ${\rm IC}_Z( \V_{\CC})$
via the Riemann-Hilbert correspondence.\footnote{A direct construction can be given, though this requires quite a bit of work 
and will not be used here. The reader interested in details can consult \cite[\S3.4]{HTT}.}
Its filtration is (nontrivially) determined by the Hodge filtration on $U$, as we will see in \S\ref{V-filtration}.

More generally, in \cite{Saito-MHP} Saito introduced an abelian category of ${\rm HM} (X, \ell)$ of pure polarizable Hodge modules on $X$ of weight $\ell$. The main two constituents of one such Hodge module $M$ are still:

\noindent
(i) ~ A filtered (regular) holonomic $\Dmod_X$-module $(\Mmod, F)$,  where $F = F_{\bullet} \Mmod$
is a good filtration by $\OX$-coherent subsheaves,  so that $\Gr_{\bullet}^F \!
\Mmod$ is coherent over $\Gr_{\bullet}^F \Dmod_X$. 

\noindent
(ii) ~ A $\QQ$-perverse sheaf $P$ on $X$ whose complexification corresponds to $\Mmod$ via the Riemann-Hilbert correspondence, so that there is an isomorphism
\[
	\alpha: \DR_X(\Mmod) \overset{\simeq}{\longrightarrow} P \tensor_{\QQ} \CC.
\]

These are subject to a list of conditions, which are defined by induction on the dimension of the support of $M$. 
If $X$ is a point, a pure Hodge module is simply a polarizable Hodge structure of weight $\ell$. In general, it is required that the nearby and vanishing cycles associated to  $M$ with respect to any locally defined holomorphic function are again Hodge modules, now on a variety of smaller dimension. This will not play a key role here, but a nice discussion can be found in \cite{Schnell-MHM}*{\S12}.

The definition of a polarization on $M$ is quite involved, but in any case involves an isomorphism ${\bf D} P \simeq P( \ell)$, 
where ${\bf D} P$ is the Verdier dual of the perverse sheaf $P$ (together of course with further properties 
compatible with the inductive definition of Hodge modules suggested above); more details in \S\ref{duality_section}.

One of the fundamental results of Saito \cite{Saito-MHP}, \cite{Saito-MHM} clarifies the picture considerably; it says that we mainly need to think of the examples 
described above as extensions of variations of Hodge structure. 
Indeed, the existence of polarizations makes the category ${\rm HM} (X, \ell)$ semi-simple: each object admits a 
decomposition by support, and simple objects with support equal to an irreducible subvariety $Z
\subseteq X$ (called pure Hodge modules with \emph{strict support} $Z$, i.e. with no nontrivial subobjects or quotient objects whose support is $Z$) are obtained from polarizable variations of Hodge structure on Zariski-open subsets of $Z$. Formally, 
\begin{equation}\label{support_decomposition}
{\rm HM}(X, \ell) = \bigoplus_{Z\subseteq X} {\rm HM}_Z (X, \ell),
\end{equation}
with ${\rm HM}_Z (X, \ell)$ the subcategory of pure Hodge modules of weight $\ell$ with strict support $Z$. In other words:

\begin{theorem}[{\bf Simple objects}, {\cite[Theorem 3.21]{Saito-MHM}}]\label{structure}
Let $X$ be a smooth complex variety, and $Z$ an irreducible closed subvariety of $X$. Then:
\begin{enumerate}
\item Every polarizable variation of Hodge structure of weight $\ell - \dim Z$ defined on a nonempty open set of $Z$ extends uniquely 
to a pure polarizable Hodge module of weight $\ell$ with strict support $Z$.
\item Conversely, every pure polarizable Hodge module of weight $k$ with strict support $Z$ is obtained in this way.
\end{enumerate}
\end{theorem}

Furthermore, M. Saito introduced in \cite{Saito-MHM} the abelian category $\MHM (X)$ 
of (graded-polarizable) mixed Hodge modules on $X$. In addition to data as in (i) and (ii) above, 
in this case a third main constituent is:

\noindent
(iii)~A finite increasing weight filtration $W_{\bullet} M$ of $M$ by
objects of the same kind, compatible with $\alpha$, such that the graded quotients
$\Gr_{\ell}^W M = W_{\ell} M / W_{\ell-1} M$ are pure Hodge modules in ${\rm HM} (X, \ell)$. 

Again, if $X$ is a point a mixed Hodge module is a graded-polarizable mixed Hodge structure, while in general
these components are subject to several conditions defined by induction on the
dimension of the support of $\Mmod$, involving
the graded quotients of the nearby and vanishing cycles of $\Mmod$.
For a further discussion of the definition see also \cite{Schnell-MHM}*{\S20}. I do not insist on giving more background 
on mixed Hodge structures and modules, as they will  be used in what follows only by reduction to pure Hodge modules.
There is however one important class of examples worth pointing out.

Let $D$ be a divisor in $X$ with complement $U$, and assume that we are given a variation of Hodge structure $\V$ on $U$.
Besides the pure Hodge module extension whose underlying 
perverse sheaf is the intersection complex ${\rm IC}_Z( \V_{\QQ})$, it is also natural to consider a mixed Hodge module 
extension, denoted $j^*j^{-1} M$ in \cite{Saito-MHM}, whose underlying perverse sheaf is simply the direct image $j_*\V_\QQ$.
More precisely, 
\[ j_*j^{-1}M = \big( (\cV(*D), F); j_*\V_\QQ\big),\]
where $\cV(*D)$ is the localization of the flat bundle 
$\cV$ underlying $\V$ along $D$, endowed with a meromorphic connection (see e.g. \cite[\S5.2]{HTT}). Further details are given in Example \ref{localization}.

Returning to the general theory, one of the most important results is M. Saito's theorem on the behavior 
of direct images of pure polarizable Hodge modules via projective morphisms. (I am only stating part of it here.)

\begin{theorem}[{\bf Stability Theorem}, \cite{Saito-MHP}*{Th\'eor\`eme 5.3.1}]\label{stability}
Let $f: X \rightarrow Y$ be a projective morphism of smooth complex varieties, and $h = c_1 (H)$ for a line bundle $H$ on $X$ which is ample relative to $f$. If $M \in {\rm HM} (X, \ell)$ is a polarizable Hodge module, then 

\noindent
(i) The filtered direct image $f_+ (\Mmod, F)$ is strict, and $H^i f_+\Mmod$ underlies a polarizable 
Hodge module $M_i  \in {\rm HM}(Y, \ell + i)$.

\noindent
(ii) For every $i$ one has an isomorphism of pure Hodge modules
$$h^i : M_{-i} \longrightarrow M_i (i).$$
\end{theorem}

As alluded to in the paragraph on strictness in \S\ref{filtered}, the statement in (i) is a key property of $\Dmod$-modules underlying Hodge modules that is not shared by arbitrary filtered $\Dmod$-modules; for more on this see e.g. \cite{Schnell-MHM}*{\S26-28}. One important consequence is  Saito's formula 
\cite[2.3.7]{Saito-MHP} giving the commutation of the graded quotients of the de Rham complex with direct images:
$$R^i f_* \Gr_k^F \DR_X(\Mmod) \simeq \Gr_k^F \DR_Y (H^i f_+ \Mmod).$$

A fundamental consequence of the theorem above deduced in \cite{Saito-MHP} is the analogue of the decomposition theorem for pure polarizable Hodge modules, obtained formally from the above as an application of Deligne's criterion for the degeneration of the Leray spectral sequence in terms of the Lefschetz operator. The result is often stated for the underlying perverse sheaves, extending the well-known BBD-decomposition theorem; here I state the filtered 
$\Dmod$-modules version, which is crucial for the applications presented later.

\begin{theorem}[{\bf Saito Decomposition Theorem}]\label{decomposition}
Let $f: X \rightarrow Y$ be a projective morphism of smooth complex varieties, 
and let $M \in {\rm HM} (X, \ell)$, with underlying filtered $\Dmod$-module $(\Mmod, F)$. Then 
$$ f_+ (\Mmod, F) \simeq \bigoplus_{i\in \ZZ} H^i f_+ (\Mmod, F) [-i]$$
in $\D^b \big({\rm FM}(\Dmod_Y)\big)$.
\end{theorem}

\begin{remark}
As we are working in the algebraic category and all mixed Hodge modules will be polarizable (cf. \cite{Saito-MHM}*{\S4.2}), I will implicitly assume that all objects are polarizable in what follows and ignore mentioning this condition.
\end{remark}

\subsection{Examples}\label{examples}
This section reviews the main examples that will be of interest in view of Saito's vanishing theorem. I will 
use freely the notation of the previous sections.

\begin{example}[{\bf The canonical bundle}]\label{basic}
If $X$ is smooth of dimension $n$ and $\V = \QQ_X$ is the constant variation of Hodge structure, we have that $P = \QQ_X[n]$,  $\Mmod = \omega_X$ with the natural right $\Dmod_X$-module structure, and $F_k \omega_X = \omega_X$ for $k \ge -n$, while $F_k \omega_X = 0$ for $k < -n$. The associated Hodge module is usually denoted $\QQ_X^H[n]$, and called the \emph{trivial 
Hodge module on $X$}. The de Rham complex of $\Mmod$ is 
\[
\DR_X(\omega_X) = \DR_X (\shO_X) [n] = \Bigl\lbrack
		\shO_X \to \Omega_X^1  \to \dotsb \to \Omega_X^n 
	\Bigr\rbrack [n].
\]
Note that 
$$\Gr_{-k}^F \DR (\omega_X, F) = \Omega^k_X [n-k]\,\,\,\,  
{\rm ~for~ all~}k.$$
Finally, in this example we have $p(\Mmod) = -n$ and $S (\Mmod) = \omega_X$.
\end{example}

\begin{example}[{\bf Direct images}]\label{direct_image}
Let $f: X \rightarrow Y$ be a projective morphism with $X$ smooth of dimension $n$ and $Y$ of dimension $m$,
and let $\V$ be a polarizable variation of $\QQ$-Hodge structure of weight $k$ on an open dense subset 
$U \subset X$, 
inducing a pure Hodge module $M$ of weight  $n + k$ on $X$ as in the previous section. If $(\Mmod, F)$ is the 
underlying filtered $\Dmod_X$-module, Theorem \ref{decomposition} gives a decomposition
$$f_+ (\Mmod, F) \simeq \bigoplus_i (\Mmod_i, F) [-i]$$ 
in the derived category of filtered $\Dmod_Y$-modules.
According to Theorem \ref{stability}, each $(\Mmod_i, F)$ underlies a 
 pure Hodge module $M_i = H^i f_*M$ on $Y$, of weight $n+k +i$. 
Furthermore, $f_+ (\Mmod, F)$ satisfies the strictness property, a particular case of which is the isomorphism
\begin{equation}\label{lowest_decomposition}
\R f_* S (\Mmod) \simeq F_{p(\Mmod)} (f_+ \Mmod) \simeq \bigoplus_i F_{p(\Mmod)} \Mmod_i [ -i]
\end{equation}
in the bounded derived category of coherent sheaves on $Y$. 

For instance, in the case when $\V = \QQ_X$ is the constant variation of Hodge structure, by Example \ref{basic}
$p (\Mmod) = -n$ and $S(\Mmod) = \omega_X$. This implies for all $i$ that
$$p(\Mmod_i) = -n \,\,\,\, {\rm and} \,\,\,\, F_{-n} \Mmod_i = R^i f_* \omega_X.$$
Note that for the corresponding left $\Dmod$-modules $\Nmod_i$ this means 
$$p(\Nmod_i) = m-n \,\,\,\, {\rm and} \,\,\,\, F_{m-n} \Nmod_i = R^i f_* \omega_{X/Y}.$$
Finally, formula (\ref{lowest_decomposition}) specializes to 
$$\R f_* \omega_X \simeq \bigoplus_i R^i f_* \omega_X [ -i],$$
which is the well-known Koll\'ar decomposition theorem \cite{Kollar2}. Moreover, we will see in Corollary \ref{lowest} and Theorem \ref{saito_GR} below that 
$R^i f_* S(\Mmod)$ satisfy other important properties known from \cite{Kollar1} in the case of canonical bundles,  like vanishing and 
torsion-freeness.
\end{example}

\begin{example}[{\bf Localization}]\label{localization}
Let $\Mmod$ be a right $\Dmod_X$-module and $D$ an effective divisor on a smooth variety $X$, given locally by an equation $f$. One can define a new $\Dmod_X$-module $\Mmod (*D)$ by localizing $\Mmod$ at $f$; in other words, globally we have
$$\Mmod(*D) =  j_* j^{-1} \Mmod,$$
where $j: U \hookrightarrow X$ is the inclusion of the complement $U = X \smallsetminus D$.
\end{example}

A standard characterization of those $\Dmod$-modules which do not change under localization will be 
useful later.

\begin{proposition}\label{equivalence_localization}
Let $X$ be a smooth complex variety, $D$ an effective divisor in $X$, and denote $j: U \hookrightarrow X$ the inclusion of the complement $U = X \smallsetminus D$. Then the restriction functor $j^*$ induces an equivalence between the following categories:

\noindent
(i) Regular holonomic  $\Dmod_X$-modules  $\Mmod$ such that the natural morphism 
$\Mmod \to \Mmod (*D)$ is an isomorphism.

\noindent
(ii) Regular holonomic $\Dmod_U$-modules.
\end{proposition}
\begin{proof}
A quick argument is to apply the Riemann-Hilbert correspondence for regular holonomic $\Dmod$-modules, see e.g. \cite[Theorem 7.2.5]{HTT}, as the condition defining the category in (i) says that for the perverse sheaf $K$ associated to $\Mmod$ one has $K \simeq j_*j^{-1} K$, i.e. $K$ can be recovered from its restriction to $U$.
\end{proof}

Assume now that $\Mmod$ underlies a mixed Hodge module $M$. By the formula above,  $\Mmod (*D)$ underlies the corresponding mixed Hodge module $j_* j^{-1} M$, and so continues to carry a natural Hodge filtration $F$. This is in general very complicated to compute; the case $\Mmod = \omega_X$, where $\omega_X (*D)$ is the sheaf of
meromorphic $n$-forms on $X$ that are holomorphic on $U$ and the corresponding Hodge module is $j_* \QQ^H_U \decal{n}$, is already very relevant. I will say a few words below, and more later.

We always have $F_{\ell} \omega_X(\ast D) \cdot F_k \Dmod_X \subseteq F_{k+\ell}
\omega_X(\ast D)$, since the filtration is compatible with the order of differential
operators, while by \cite{Saito-b}*{Proposition~0.9} we have
\begin{equation} \label{eq:FP}
	F_k \omega_X(\ast D) \subseteq  P_k \omega_X(\ast D) = 
\begin{cases}
\omega_X \bigl( (n+ k+1) D \bigr)  & \mbox{if} \,\,\,\, k \ge -n \\ 
0 &\mbox{if} \,\,\,\, k < -n,
\end{cases}
\end{equation}
i.e. the Hodge filtration is contained in the filtration by pole order. Furthermore, in \cite{Saito-b}*{Corollary 4.3} it is 
shown that if $D$ is smooth, then 
$$F_k \omega_X(\ast D) = P_k \omega_X(\ast D) \,\,\,\, {\rm ~for~ all~} k.$$
In general, a detailed analysis of the Hodge filtration on $\omega_X(*D)$ is given in the upcoming \cite{MP}.

We will see in Section \ref{particular_cases} that the first nontrivial step in the filtration is always related to the 
$V$-filtration along $D$, and that this provides a useful relationship with multiplier ideals. For this purpose it is more convenient to write things in terms of left $\Dmod$-modules. In fact, for the left 
$\Dmod$-module $\OX (*D)$ associated to $\omega_X (*D)$ (recall that $F_p \OX (*D) = F_{p -n} \omega_X (*D) \otimes \omega_X^{-1}$), one has the formula
$$S \big( \shO_X (*D) \big) = F_0 \OX(\ast D) = {\tilde V}^1 \OX \cdot \OX(D).$$
The $V$-filtration on $\Mmod$ and $\Mmod(*D)$ is discussed in \S\ref{V-filtration}, and provides further insight into the process of localization.

\subsection{Duality}\label{duality_section}
For later use, a few words are in order about duality for polarized Hodge modules, on a  smooth projective 
variety $X$ of dimension $n$. Further discussion 
and references can be found for instance in \cite{Schnell-MHM}*{\S13 and \S29}.

As mentioned earlier, a polarization on a 
pure Hodge module $M = \big((\Mmod, F); P\big)$ of weight $\ell$ involves an isomorphism $P(\ell) \simeq \D P$, where $\D P$ is the Verdier dual of the perverse sheaf $P$, compatible with the filtration $F$. This means that for the dual holonomic right 
$\Dmod_X$-module
$$\D \Mmod := \mathcal{E}xt^n (\Mmod, \omega_X \otimes_{\OX} \Dmod_X)$$
we have $\D \Mmod \simeq \Mmod$, but furthermore the natural induced filtration on $\D \Mmod$ 
should satisfy
$$(\D \Mmod, F) \simeq (\Mmod, F_{\bullet - \ell}).$$
It is necessary therefore for the filtration on $\D \Mmod$ to be strict. In fact, it is standard that this strictness property is equivalent to  
$\Gr_{\bullet}^F \Mmod$ being Cohen-Macaulay as a $\Gr_{\bullet}^F \Dmod_X$-module; this last statement holds by \cite{Saito-MHP}*{Lemma 5.1.13} for filtered $\Dmod$-modules underlying Hodge modules. A consequence is that one can define a dual Hodge module $\D M$, and in fact
$\D M \simeq M(\ell)$, with underlying filtered $\Dmod$-module $(\Mmod, F_{\bullet - \ell})$.

Moreover, by \cite{Saito-MHP}*{\S2.4.11} the filtered de Rham complex commutes with the duality functor. Given the discussion above, a useful consequence is:

\begin{lemma}\label{duality}
If $X$ is a smooth projective variety of dimension $n$ and  $(\Mmod, F)$ is the filtered $\Dmod$-module underlying a pure Hodge module $M \in {\rm HM} (X, \ell)$, then 
$$\R \Delta \Gr_k^F \DR_X (\Mmod) \simeq \Gr_{-k- \ell}^F \DR_X (\Mmod),$$
where $\R \Delta = \R \mathcal{H}om_{\OX} (\,\cdot \, , \omega_X) [n]$ is the Grothendieck duality functor.
\end{lemma}

\subsection{The $V$-filtration.}\label{V-filtration}
In this section I will recall some key definitions and results regarding the $V$-filtration with respect to a hypersurface, and its interaction with the Hodge filtration. I am mostly following \cite{Saito-MHP}*{\S3}, which is a complete reference for all the definitions and results recalled here.

Let $X$ be a complex manifold or smooth complex variety of dimension $n$, and let $X_0$ be an smooth divisor on $X$ defined locally by an equation $t$. We first consider a rational filtration on $\Dmod_X$, given by
$$V_{\alpha} \Dmod_X = \{ P \in \Dmod_X~|~ P \cdot \I_{X_0}^j \subseteq \I_{X_0}^{j - [\alpha]}\} \,\,\,\,{\rm for~} 
\alpha \in \QQ,$$
where $\I_{X_0}$ is the ideal of $X_0$ in $\shO_X$, with the convention that $\I_{X_0}^j  = \shO_X$ for $j \le 0$.

\begin{definition}[{\bf $V$-filtration}]
Let $\Mmod$ be a coherent right $\Dmod_X$-module. A \emph{rational $V$-filtration} (a slight refinement of the \emph{Kashiwara-Malgrange filtration}) of $\Mmod$ 
along $X_0$ is an increasing filtration $V_{\alpha} \Mmod$ with $\alpha \in \QQ$ satisfying the 
following properties:

\noindent
$\bullet$\,\,\,\,The filtration is exhaustive, i.e. $\bigcup_{\alpha} V_{\alpha} \Mmod = \Mmod$, and each $V_{\alpha} \Mmod$
is a coherent $V_0 \Dmod_X$-submodule of $\Mmod$.

\noindent
$\bullet$\,\,\,\,$V_\alpha \Mmod \cdot V_i \Dmod_X \subseteq V_{\alpha + i } \Mmod$ for every $\alpha \in \QQ$ and 
$i \in \ZZ$; furthermore
$$V_{\alpha} \Mmod \cdot t = V_{\alpha-1} \Mmod \,\,\,\, {\rm for ~} \alpha < 0.$$

\noindent
$\bullet$\,\,\,\,The action of $t\partial_t -\alpha$ on $\Gr_{\alpha}^V \Mmod$ is nilpotent for each $\alpha$, where 
$\partial_t$ is a vector field such that $[\partial_t, t] =1$. (One defines $\Gr_{\alpha}^V \Mmod$ as $V_{\alpha} \Mmod / V_{< \alpha} \Mmod$, 
where $V_{< \alpha} \Mmod = \cup_{\beta < \alpha} V_{\beta} \Mmod$.)
\end{definition}

It is known that if a $V$-filtration exists, then it is unique. In addition, $\Dmod$-modules underlying mixed Hodge modules also come by definition with a Hodge filtration, and it is important to compare the two. Note first that on each 
$ \Gr_{\alpha}^V \Mmod$ one considers the filtration induced by that on $\Mmod$, i.e.
$$F_p \Gr_{\alpha}^V \Mmod := \frac{F_p \Mmod \cap V_{\alpha} \Mmod}{F_p \Mmod \cap V_{< \alpha} \Mmod}.$$

\begin{definition}[{\bf Regular and quasi-unipotent}]\label{quasi-unipotent}
In the situation above, assume that $\Mmod$ is endowed with a good filtration $F$.
We say that $(\Mmod, F)$ is \emph{quasi-unipotent (or strictly specializable) along $X_0$}
if $\Mmod$ admits a rational $V$-filtration along $X_0$ and the following conditions are satisfied:

\noindent
$\bullet$~ $(F_p V_{\alpha} \Mmod)\cdot t = F_p V_{\alpha - 1} \Mmod$ \,\, for  \,\, $\alpha < 0$.

\noindent
$\bullet$~ $(F_p \Gr_{\alpha}^V \Mmod)\cdot \partial_t = F_{p+1} \Gr_{\alpha+1}^V \Mmod$  \,\, for  \,\, 
$\alpha > -1$.

\noindent
One says that $(\Mmod, F)$ is \emph{regular and quasi-unipotent along $X_0$} if in addition the filtration 
$F_{\bullet} \Gr_{\alpha}^V \Mmod$ is a good filtration for $-1 \le \alpha \le 0$.
\end{definition}

\noindent
Let now $f: X \rightarrow \CC$ be a holomorphic function, and denote by 
$$i = i _{\Gamma_f}: X \hookrightarrow X\times \CC =  Y$$
the embedding of $X$ as the graph of $f$. Denote by $t$ the coordinate on $\CC$, so that in the notation above we
have $X_0 = X \times \{0 \} = t^{-1} (0)$.
If $\Mmod$ is a coherent right $\Dmod_X$-module, denote $(\tilde \Mmod, F) = i_* (\Mmod, F)$.
One says that $(\Mmod, F)$ is \emph{strictly specializable along $f$} if $(\tilde \Mmod, F)$ is so along $X_0$, and the 
same for \emph{regular and quasi-unipotent along $f$}. One important feature of mixed Hodge module theory is that all 
$\Dmod$-modules underlying Hodge modules are required to satisfy this last property with respect to any holomorphic function. 

The following technical result on the behavior of regular and quasi-unipotent filtered $\Dmod$-modules is a key step in extending Kashiwara's theorem on closed embeddings to the setting of Hodge $\Dmod$-modules. This will be very useful when stating Saito's vanishing theorem on singular varieties in Section \ref{main}. 

\begin{lemma}[\cite{Saito-MHM}*{Lemma 3.2.6}]\label{divisor_support}
Let $f: X\rightarrow \CC$ be a holomorphic function, and $(\Mmod, F)$ a filtered coherent $\Dmod_X$-module. Assume that ${\rm Supp}(\Mmod) \subseteq
f^{-1}(0)$. Then  the following are equivalent:

\noindent
(i) $(\Mmod, F)$ is regular and quasi-unipotent along $f$.

\noindent
(ii) $\Gr^F_p \Mmod \cdot f =0$  for all $p$.

\noindent
(iii) $(\tilde M, F) \simeq j_* (\Mmod, F)$, where $j : X \times \{0\} \hookrightarrow X \times \CC$.

\end{lemma}

We will also need a transversality notion for a  filtered $\Dmod$-module with respect to a morphism (or a 
submanifold) introduced in \cite{Saito-MHP}*{3.5.1}, under which filtered  inverse images become particularly simple.

\begin{definition}[{\bf Non-characteristic morphism}]\label{transverse}
Let $f: X \rightarrow Y$ be a morphism of complex manifolds, and let $(\Mmod, F)$ be a filtered coherent $\Dmod_Y$-module.
One says that $f$ is \emph{non-characteristic} for $(\Mmod, F)$ if the following two conditions are satisfied:

\noindent
$\bullet$\,\,\,\,$H^i (f^{-1} \Gr^F \Mmod \overset{{\bf L}}{\otimes}_{f^{-1} \shO_Y} \shO_X) = 0$ for $i \neq 0$.\footnote{In terms of the individual graded pieces, which are coherent sheaves of $Y$, this simply says that  $L^i f^* \Gr_k^F \Mmod = 0$ for all $i \neq 0$ 
and all $k$.}

\noindent
$\bullet$\,\,\,\,The natural morphism 
$$df^*: p_2^{-1} \big({\rm Char} (\Mmod) \big) \rightarrow T^*X$$ 
is finite, where $p_2: X \times_Y T^*Y \rightarrow T^*Y$ is the second projection and 
$$df^*: X\times_Y T^*Y \rightarrow T^*X, \,\,\,\,(x, \omega) \mapsto df^* \omega \,\,\,\,{\rm for~all~} x\in X,\,\, \omega \in T^* Y.$$
\noindent
If $f$ is a closed immersion, we say that $X$ is non-characteristic for $(\Mmod, F)$ if $f$ is so.
\end{definition}

If $f$ is non-characteristic for $(\Mmod, F)$ and $d = \dim X - \dim Y$, then as in \cite{Saito-MHP}*{\S3.5} one has the filtered pullback 
$\mu^*(\Mmod, F) = (\tilde \Mmod, F) [-d]$ given by the formula 
$$\tilde \Mmod = \mu^{-1} \Mmod \otimes_{\mu^{-1} \shO_X} \omega_{Y/X} \,\,\,\,{\rm and} \,\,\,\, 
F_p \tilde \Mmod = \mu^{-1} F_{p+d} \Mmod \otimes_{\mu^{-1} \shO_X} \omega_{Y/X}.$$
In other words we can define the inverse image to be, up to shift, the naive filtration on the naive pullback, and this again gives a holonomic $\Dmod_X$-module if $\Mmod$ is so; see \cite{Saito-MHP}*{Lemma 3.5.5}. If $(\Mmod, F)$ underlies a pure Hodge module $M$ of weight $m$ on $Y$, $\mu^*(\Mmod, F)$ then underlies a pure Hodge module $\mu^* M$ of weight $m + d$ on $X$.\footnote{This is a non-trivial result, using the fact that pure Hodge modules with strict support come from generic variations of Hodge structure; see e.g. \cite{Schnell-MHM}*{\S30} for an explanation.}

\begin{example}\label{nonchar_example}
The most basic examples of this notion are:

\noindent
(i) If $f$ is smooth, then it is non-characteristic for any $(\Mmod, F)$, as $df^*$ is injective and $f$ is flat. 

\noindent
(ii) If $(\Mmod, F)$ underlies a variation of Hodge structure, any $f$ is non-characteristic for it, as 
${\rm Char}(\Mmod)$ is the zero section, while each $\Gr_k^F \Mmod$ is locally free.

As a combination of the two, if $f$ is smooth outside of the locus where $(\Mmod, F)$ underlies a variation of Hodge structure, then $f$ is non-characteristic for $(\Mmod, F)$.
\end{example}

The following lemmas are important in what follows; they show that under the non-characteristicity assumption
one can perform concrete calculations with the $V$-filtration.

\begin{lemma}[\cite{Saito-MHP}*{Lemme 3.5.6}]\label{noncharacteristic}
Let $i: D \hookrightarrow X$ be an inclusion of a smooth hypersurface in a smooth complex variety. Let $(\Mmod, F)$ be a filtered coherent 
right $\Dmod_X$-module for which $D$ is non-characteristic. Then

\noindent
(1) $(\Mmod, F)$ is regular and quasi-unipotent along $D$.

\noindent
(2) The $V$-filtration on $M$ is given by 
$$V_{\alpha} \Mmod = \Mmod \cdot \shO_X (- i D) \,\,\,\,{\rm for~} -i-1 \le \alpha < -i, \,\,i \ge 0 \,\,\,\,{\rm and~} V_{\alpha} \Mmod = \Mmod \,\,
{\rm for~}\alpha \ge 0.$$
\end{lemma}

\begin{lemma}[\cite{Saito-MHP}*{Lemme 3.5.7}]\label{local_nonchar}
With the notation of Lemma \ref{noncharacteristic}, we have that

\noindent
(1) The $V$-filtration on $\Mmod(*D)$ satisfies
$$V_{\alpha} \Mmod (*D) = \Mmod \cdot \shO_X (- i D) \,\,\,\,{\rm for~} -i-1 \le \alpha < -i.$$ 

\noindent
(2) There is a filtration $F$ on $\Mmod (*D)$ which makes it a filtered coherent right $\Dmod_X$-module,
such that there is an exact sequence of filtered $\Dmod$-modules
$$0\longrightarrow (\Mmod, F) \longrightarrow (\Mmod (*D), F) \longrightarrow i_* i^! (\Mmod, F)[1] \longrightarrow 0.$$
In addition, $(\Mmod (*D), F)$ is regular and quasi-unipotent along $D$.
\end{lemma}

It will also be crucial, under suitable hypotheses, to be able to recover the Hodge filtration from its restriction over 
the complement of a hypersurface. This is one of the key points of the interaction between the Hodge filtration and 
the $V$-filtration in the case of filtered $\Dmod$-modules underlying Hodge modules.

\begin{lemma}[\cite{Saito-MHP}*{Proposition 3.1.8}]\label{extension_filtration}
With the notation of Lemma \ref{noncharacteristic}, and $U = X \smallsetminus D$, let $\Mmod^\prime$ be the smallest sub-object of $\Mmod$ such that 
$\Mmod_{|U} = \Mmod^\prime_{|U}$. Then:

\noindent
(1) $\Mmod^\prime = V_\alpha \Mmod \cdot \Dmod_X$ for $\alpha < 0$.

\noindent
(2) $\Mmod / \Mmod^\prime \simeq i_* {\rm Coker}~\big( {\rm can} = \partial_t : \Gr_{-1}^V \Mmod \rightarrow \Gr_0^V\Mmod\big)$.

\noindent
In particular, $\Mmod = V_\alpha \Mmod \cdot \Dmod_X$ for $\alpha < 0$ if ${\rm can}$ is surjective.
\end{lemma}

\begin{lemma}[\cite{Saito-MHP}*{Proposition 3.2.2}]\label{equivalence}
With the notation of Lemma \ref{extension_filtration}, and $j: U \rightarrow X$ the natural inclusion, we have that:

\noindent
(1) The first condition in Definition \ref{quasi-unipotent} is equivalent to 
$$F_p V_{< 0} \Mmod = V_{< 0 } \Mmod \cap j_* j^{-1} F_p \Mmod \,\,\,\,{\rm for ~all~}p.$$

\noindent
(2) If $\Mmod = V_{< 0} \Mmod \cdot \Dmod_X$, or equivalently if $ {\rm can} = \partial_t : \Gr_{-1}^V \Mmod \rightarrow \Gr_0^V\Mmod$  is 
surjective, the second condition in Definition \ref{quasi-unipotent} for $\alpha \ge -1$\footnote{There is an extra point here, 
for which I am grateful to C. Sabbah: in Definition
\ref{quasi-unipotent} one only considers $\alpha  > -1$, while in the lemma $\alpha = -1$ appears 
as well. However, the property we want for $\alpha = -1$ follows from Hodge theory conditions on $\Gr^V_{-1}$ and 
$\Gr^V_0$; in our application they will be trivially satisfied since both terms will be $0$.}  is equivalent to 
$$F_p \Mmod = \sum_{i \ge 0} (F_{p-i}  V_{< 0} \Mmod)\cdot \partial_t^i \,\,\,\,{\rm for ~all~} p.$$
\end{lemma}

\subsection{Kodaira-Saito vanishing}\label{main}
We now come to the main goal, M. Saito's vanishing theorem. Before stating and proving the theorem, it is important to emphasize the following point: this is a result that works on \emph{singular} varieties by embedding them into smooth ambient spaces. It is known that the objects considered are independent of the embedding.

It is therefore important to have a way of thinking about mixed Hodge modules and filtered $\Dmod$-modules on singular 
varieties, compatible with the material developed for smooth varieties. In general this can only be done be locally embedding $X$ into smooth ambient spaces, and then using a gluing procedure (see \cite{Saito-MHM}*{\S2.1}).

However, on projective varieties we can use the embedding 
of $X$ into some $\PP^N$. If $X\hookrightarrow \PP^N$ is one such, then one defines the category of mixed Hodge modules on $X$ to be that of mixed Hodge modules on $\PP^N$ with support contained in $X$, i.e.
$${\rm MHM} (X) = {\rm MHM}_X (\PP^N).$$
One can do the same with any embedding $X\subset Z$ into a smooth variety; at least when $Z$ is projective, the fact that the resulting MHM($X$) is independent of the embedding follows by extending Kashiwara's equivalence theorem for closed embeddings to the setting of Hodge modules.

Indeed, recall that Kashiwara's theorem says that for a closed embedding $h:Z \hookrightarrow W$ one has
$${\rm Mod}_{{\rm coh}} (\Dmod_Z) \simeq {\rm Mod}_{{\rm coh}, Z} (\Dmod_W),$$
where the category on the right is that of coherent $\Dmod_W$-modules with support contained in $Z$. This correspondence restricts on both sides to the subcategories of objects with support contained in $X$. The equivalence does not extend in general to filtered $\Dmod$-modules; however, those underlying mixed Hodge modules are regular and quasi-unipotent (Definition \ref{quasi-unipotent}) along the zero-locus of any holomorphic function. 

In the regular and quasi-unipotent case, one can use Lemma \ref{divisor_support} for each local defining equation $f$ for $Z$ inside $W$ (or global equations when $W = \PP^N$) in order to deduce that for every $(\Mmod, F)$ on $W$ with support in $Z$, there exists $(\Mmod_Z, F)$ on $Z$ such that $(\Mmod, F) \simeq h_* (\Mmod_Z, F)$. Thus Kashiwara's theorem extends to these special filtered holonomic $\Dmod$-modules, which is the key step in extending it to mixed Hodge modules. Once this is established, it is not too hard to deduce that MHM($X$) 
is independent of the embedding; formally
\begin{equation}\label{kashiwara}
{\rm HM} (X, \ell) = {\rm HM}_X (Z, \ell) \,\,\,\,{\rm and} \,\,\,\, \MHM (X) = \MHM_X (Z)
\end{equation}
for any smooth $Z$ containing $X$. Further details can be found in \cite{Saito-MHP}*{Lemme 5.1.9}
and \cite{Saito-MHM}*{2.17.5}; see also \cite{Schnell3}*{\S6 and 7}.

\begin{theorem}[M. Saito, \cite{Saito-MHM}*{\S2.g}]\label{saito_vanishing}
Let $X$ be a complex projective variety, and $L$ an ample line bundle on $X$. Consider an integer $m >0 $ such that $L^{\otimes m}$ is very ample 
and gives an embedding $X \subseteq \PP^N$. Let $(\Mmod, F)$ be the filtered $\Dmod$-module underlying a mixed Hodge module $M$ on $\PP^N$ with support contained in $X$, i.e. an object in $\MHM (X)$. Then:

\noindent
(1) $\Gr_k^F \DR_{\PP^N} (\Mmod)$ is an object in $\D^b (X)$ for each $k$, independent of the embedding of $X$ in $\PP^N$.\footnote{In fact, based 
on the discussion above it can be shown that each $\Gr_k^F \DR_{\PP^N} (\Mmod)$ is independent of the embedding of $X$ into any smooth complex variety.}

\noindent
(2) We have the hypercohomology vanishing 
$$H^i \bigl( X, \Gr_k^F \DR_{\PP^N} (\Mmod) \tensor L \bigr) = 0 {\rm ~ for~ all~} i > 0.$$
and 
$$H^i \bigl( X, \Gr_k^F \DR_{\PP^N} (\Mmod) \tensor L^{-1} \bigr) = 0\,\,\,\,{\rm~ for~ all~} i < 0.$$
\end{theorem}
\begin{proof}

\noindent
\emph{Step 1.}
This step addresses (1) and a number of useful reductions towards (2). For the first statement in (1), due 
to the definition of $\Gr_k^F \DR_{\PP^N} (\Mmod)$, it is enough to have that each $\Gr_k^F \Mmod$ is an $\shO_X$-module. But note 
that by Lemma \ref{divisor_support}, if for a holomorphic function $f$ the support of $\Mmod$ is contained in $f^{-1} (0)$, 
the condition of $(\Mmod, F)$ being regular and quasi-unipotent along $f$ is equivalent to having 
$$\Gr^F_p \Mmod \cdot f =0  \,\,\,\,\,\,{\rm for ~all~}p.$$
Now our $(\Mmod, F)$ satisfies this for any $f$, as it underlies a Hodge module, and applying it for the defining equations of $X$ inside $\PP^N$
we obtain the conclusion. 

Note that the independence on the embedding of the definition $\MHM (X) = \MHM_X (\PP^N)$ 
follows from the discussion preceding the statement of the theorem. However here strictly speaking  
one only needs to know independence of embeddings $X \hookrightarrow \PP^N$ by various powers $L^{\otimes m}$. 
Thus the Kashiwara-type statement
($\ref{kashiwara}$) actually suffices, as any two such can be compared inside a common Veronese embedding.

Along the same lines, the independence of the embedding for the complex of $\shO_X$-modules $\Gr_k^F \DR_{\PP^N} (\Mmod)$ 
follows then from the remark above and the fact that if $h:Z \hookrightarrow W$ is a closed embedding of two smooth varieties containing $X$, and 
$(\Mmod, F) \simeq h_* (\Mmod_Z, F)$ on $W$, then one has the easily checked formula
$$\Gr_k^F \DR_{W} (\Mmod, F )\simeq h_* \Gr_k^F \DR_{Z} (\Mmod_Z, F).$$

Based on the fact that our objects do not depend on the embedding $X \subseteq \PP^N$, to attack (2) we may assume furthermore that $m \ge 2$. This will come up later, as we will need to produce non-integral rational numbers with denominator $m$.

A standard reduction is that it is enough to assume that $M$ is a polarized pure Hodge module with strict support $X$, of some weight $d$. First, once we have reduced to the case of pure Hodge modules, we can apply the strict support direct sum decomposition ($\ref{support_decomposition}$) to reduce to this case. On the other hand, if $M$ is in $\MHM (X)$, recall that it has a 
finite weight filtration $W_{\bullet} M$ by objects in $\MHM (X)$, such that the graded quotients
$\Gr_{\ell}^W M = W_{\ell} M / W_{\ell-1} M$ are in ${\rm HM} (X, \ell) = {\rm HM}_X (\PP^N, \ell)$. 
To reduce to the pure case, we simply 
use the fact that the functor  $\Gr_k^F \circ \DR$ is exact by construction.
 
Given this last reduction, we also see that it is enough to check only the second statement in (2). This follows from Grothendieck-Serre duality and Lemma \ref{duality}.

\noindent
\emph{Step 2.}
Let $Y$ be a general hyperplane in $\PP^N$, chosen to be  
non-characteristic for $(\Mmod, F)$. Denote $D = X \cap Y$, the zero locus of some section 
$s \in H^0 (X , L^{\otimes m})$.
Let $f: \tilde X \rightarrow X$ be the $m$-fold cyclic cover branched along $D$ as 
in Proposition \ref{cyclic}, with $f^* D = mD^\prime$ and $L^\prime = \shO_{\tilde X} (D^\prime)$. 

Denote now 
$$U= \PP^N \smallsetminus Y \,\,\,\,{\rm and} \,\,\,\, j : U \hookrightarrow \PP^N$$
the natural inclusion of the (affine) complement of $Y$.
Denoting also by $i : Y\hookrightarrow \PP^N$ the inclusion of $Y$, by Lemma \ref{local_nonchar} there is a filtered short exact sequence
\begin{equation}\label{poles}
0 \longrightarrow (\Mmod, F) \longrightarrow (\Mmod(*Y), F) \longrightarrow (\shH^1 i^{!} \Mmod, F) \longrightarrow 0
\end{equation}
(Note that here $\shH^1 i^{!} \Mmod$ simply means $\Mmod \otimes \omega_{Y/\PP^N}$.) 

For each $k$, we apply the exact functor $\Gr_k^F \circ \DR_{\PP^N}$ to (\ref{poles}) to obtain a distinguished 
triangle of complexes of coherent sheaves on $X$: 
$$ \Gr_k^F \DR_{\PP^N}(\Mmod)\otimes L^{-1} \longrightarrow \Gr_k^F \DR_{\PP^N} \big(\Mmod(*Y)\big)\otimes L^{-1}  \longrightarrow$$ 
$$\longrightarrow \Gr_k^F \DR_{\PP^N} (\shH^1 i^{!} \Mmod)\otimes L^{-1} \longrightarrow  \Gr_k^F \DR_{\PP^N}(\Mmod)\otimes L^{-1} [1] .$$
The claim is that 
\begin{equation}\label{claim}
H^i \big(X, \Gr_k^F \DR_{\PP^N}\big(\Mmod(*Y)\big)\otimes L^{-1} \big) = 0 \,\,\,\,{\rm for~all}~i \neq 0.
\end{equation}
This will be proved in Step 4.
Assuming it for now, by the long exact sequence on cohomology we are reduced to showing
$$H^i \big(X, \Gr_k^F \DR_{\PP^N} (\shH^1 i^{!} \Mmod)\otimes L^{-1} \big) = 0 \,\,\,\,{\rm for~all}~i < -1.$$
But in fact the statement is true even for $i < 0$ by induction on $n = \dim X$, since $(\shH^1 i^{!} \Mmod, F)$ is supported on $D$
and, again by non-characteristic pullback as in Section \ref{V-filtration}, it underlies a Hodge module in ${\rm HM}_D (Y, d+1)$.

\noindent
\emph{Step 3.}
Note first that we can extend the cover $f: \tilde X \rightarrow X$ ramified over $D$
to a cover still denoted $f: \tilde \PP^N \rightarrow \PP^N$, ramified over $Y$; it is enough 
to do this locally since Hodge modules are local by construction. Fix a point $x \in X$. The claim is that there exists a 
neighborhood $x \in U_x \subset \PP^N$ such that the restriction of $f: \tilde X \rightarrow X$ over $U_x \cap X$ can be extended to a finite cover $f_x: V_x \rightarrow U_x$, ramified over $Y \cap U_x$. If $x\not\in Y$, it is clear that there is such an extension. On the other hand, if $x \in Y$, then one uses 
a local holomorphic trivialization $(U_x, U_x \cap X) \simeq (U_x \cap Y, U_x \cap D) \times {\mathbb D}^2$, considering a contractible neighborhood of $x$ in $Y$ such that the contraction is compatible with $D$.

This new $f$ is non-characteristic for $(\Mmod, F)$ by our choice of $Y$, and so the filtered pullback $(f^*\Mmod, F)$ 
on $\tilde \PP^N$ can be defined as in the remarks after Definition \ref{transverse}. It underlies a pure Hodge module $f^*M$ of weight $d$, as the relative dimension is zero. By Theorem \ref{stability} we then obtain $f_* f^* M \in {\rm HM}_X (\PP^N, d)$; note that this is a single Hodge module since $f$ is finite. 
There is a natural monomorphism $M \rightarrow f_* f^* M$, and we define $\tilde M$ as its cokernel, so that there
is an exact sequence
\begin{equation}\label{push-pull}
0\longrightarrow M \longrightarrow f_*f^* M\longrightarrow \tilde M \longrightarrow 0,
\end{equation}
in the abelian category ${\rm HM}_X (\PP^N, d)$, i.e. $\tilde M$ is a new pure polarized Hodge module 
of weight $d$ with support contained in $X$. Note that by Saito's fundamental result mentioned in Section \ref{VHS}, all the 
Hodge modules in the exact sequence above are uniquely extended from the open subset of $U$ on which they are variations of Hodge structure; in particular they coincide with the strict support extension of their 
restriction to $U$.


We denote by $\tilde P$ the $\QQ$-perverse sheaf associated to $\tilde M$, so that 
$\DR_{\PP^N} (\tilde \Mmod) \simeq \tilde P_{\CC} := \tilde P \otimes \CC$. Since as mentioned above $\tilde M$ is the unique extension with
strict support $X$ of its restriction to $U$,  we have 
$$\tilde P \simeq j_* j^{-1} \tilde P,$$
i.e. $\tilde P$ is the extension of its restriction to the affine open set $U$ as well. By the Artin-Grothendieck vanishing theorem 
(see e.g. \cite{Lazarsfeld}*{Theorem 3.1.13}), we then have
$$H^i (X, \tilde P_{\CC}) \simeq H^i (U, j^{-1} \tilde P_{\CC}) = 0 \,\,\,\,{\rm for~all~} i > 0.$$
Since $\tilde M$ is polarized, as in Section \ref{duality} we have that $\D \tilde P \simeq \tilde P (d)$, 
where $\D \tilde P$ is the Verdier dual. By Verdier duality we then also get 
$$H^i (X, \tilde P_{\CC})  = 0 \,\,\,\,{\rm for~all~}\,\, i < 0.$$
In conclusion, we have verified that 
\begin{equation}\label{vanishing_1}
 H^i \big(X, \DR_{\PP^N} (\tilde \Mmod)\big) = 0 \,\,\,\, {\rm for~all~} \,\, i \neq 0.
\end{equation}

The main assertion in this step is that
\begin{equation}\label{vanishing_2}
H^i \big(X, \Gr_k^F \DR_{\PP^N} (\tilde \Mmod) \big) = 0 \,\,\,\,{\rm for~all}~k~{\rm and~all~} i \neq 0.
\end{equation}
To this end we need to use stability under projective morphisms, Theorem \ref{stability}; applied to the absolute case $\PP^N \rightarrow {\rm pt}$, the strictness in the statement amounts to the degeneration at $E_1$ of 
the natural Hodge-to-de Rham spectral sequence  
$$E_1^{p,q} = H^{p+ q}  \big(X, \Gr_{-q}^F \DR_{\PP^N} (\tilde \Mmod) \big) \implies  H^{p+q} \big(X, \DR_{\PP^N} (\tilde \Mmod)\big).$$
Note that here we are using the identification $f_*  \tilde \Mmod \simeq \R \Gamma  \DR_{\PP^N} (\tilde \Mmod)$ (which is a special 
case of the definition of push-forward via smooth morphisms). Given (\ref{vanishing_1}), this degeneration immediately implies (\ref{vanishing_2}).

\noindent
\emph{Step 4.} We are left with proving (\ref{claim}), which will be done in this step.
More precisely, for each $k$ we will prove the isomorphism 
\begin{equation}\label{main_isom}
\Gr_k^F \tilde \Mmod \simeq \Gr_k^F \Mmod(* Y) \otimes \tilde L,
\end{equation}
where 
$$\tilde L : = {\rm Coker} (\shO_X \rightarrow \pi_* \shO_{\tilde X}) \simeq L^{-1} \oplus \cdots\oplus L^{- (m-1)},$$
the last isomorphism coming from Proposition \ref{cyclic}. The isomorphism (\ref{main_isom}) implies what we want in combination with (\ref{vanishing_2}); it is proved using the interaction between the Hodge filtration and the $V$-filtration along $Y$.

To this end, note first that by definition there is a canonical isomorphism of filtered right $\Dmod_U$-modules
$$(\tilde \Mmod, F)_{|U} \simeq (\Mmod, F)\otimes_{\shO_U} \tilde L_{|U}.$$
Indeed, this follows from (\ref{push-pull}) and the definition of the filtration on $f^*\Mmod$ given after Definition
\ref{transverse}; passing to the filtration on the $\Dmod$-modules underlying (\ref{push-pull}) is, on the open set $U$ on which $f$ is \'etale, the same as the split short exact sequence 
$$0 \longrightarrow F_p \Mmod \longrightarrow F_p \Mmod \otimes f_* \shO_{\tilde \PP^N} \longrightarrow F_p \Mmod \otimes \tilde L \longrightarrow 0.$$

Here and in what follows we consider $\tilde L$  as a left $\Dmod$-module with trivial filtration. On the open set $U$ it is by definition an integrable connection, underlying the complement of $\QQ_U$ in $f_* \QQ_{f^{-1}(U)}$. On the other hand, we know from \cite[\S6]{EV}
that globally each $L^{-i}$ is the Deligne canonical extension of $L^{-i}_{|U}$, whose meromorphic connection has residue $i/m$ along $Y$. The direct sum $\tilde L$ is the $\Dmod$-module underlying the canonical extension of this complement.
The tensor product $\Mmod \otimes \tilde L$ becomes a right $\Dmod$-module,\footnote{Recall that if $\Mmod$ is a right $\Dmod_X$-module, and $\Nmod$ a left $\Dmod_X$-module, the tensor product $\Mmod \otimes_{\shO_X} \Nmod$ has a natural 
right $\Dmod_X$-module structure; see e.g. \cite[Proposition 1.2.9(ii)]{HTT}.} with the induced tensor product filtration.

The statement follows if we show that the isomorphism on $U$ above can be extended uniquely to an isomorphism of filtered 
right $\Dmod$-modules
\begin{equation}\label{extension}
(\tilde \Mmod, F) \simeq (\Mmod (*Y), F)\otimes_{\shO_{\PP^N}} \tilde L.
\end{equation}
Both sides of (\ref{extension}) are regular holonomic; moreover, they
are isomorphic to their localization along $Y$, i.e. a local equation of $Y$ acts on them bijectively.
Forgetting about the filtration, the 
isomorphism in (\ref{extension}) then follows from Proposition \ref{equivalence_localization}. 

As for the filtration $F$, we need to compare it to the $V$-filtration along the divisor $Y$. 
The first claim is that $(\Mmod (*Y), F)\otimes \tilde L$ is regular and quasi-unipotent along $Y$; see Definition \ref{quasi-unipotent}.
To this end, note first that the summand $\Mmod (*Y) \otimes L^{-i}$ of $\Mmod (*Y) \otimes \tilde L$, 
with $i$ ranging from $1$ to $m-1$, coincides with $\Mmod (*Y)$ on $U$, and so their $V$-filtrations along $Y$ are the same 
for $\alpha < 0$.  On the other hand, if $t$ is a local equation for $Y$, then multiplication by $L^{-i}$ coincides with the action of 
$t^{i /m}$, and so by the definition of the $V$-filtration we obtain for each $\alpha$:
\begin{equation}\label{V_commutation}
V_\alpha \big(\Mmod (*Y) \otimes L^{-i}\big) = V_{\alpha + i/m} \Mmod (*Y) \otimes L^{-i} .
\end{equation}
This gives in particular  
$$
F_p V_{\alpha} \big(\Mmod (*Y) \otimes L^{-i} \big) = \big(F_p V_{\alpha + i/m} \Mmod (*Y) \big) \otimes L^{-i}
$$
for all possible indices. Using this, the fact that $(\Mmod (*Y), F)\otimes \tilde L$ is regular and quasi-unipotent is an immediate consequence of the fact that $\Mmod (*Y)$ is so (as it underlies a mixed Hodge module), together with Lemma \ref{local_nonchar}(1).

From ($\ref{V_commutation}$) we also obtain that
$$\Gr_{\alpha }^V \big(\Mmod (*Y) \otimes L^{-i} \big) \simeq \Gr_{\alpha + i/m}^V \big(\Mmod (*Y)\big) \otimes L^{-i}.$$
We see however from Lemma \ref{local_nonchar}(1) that 
$$\Gr_{\alpha}^V \Mmod (*Y) = 0 \,\,\,\,{\rm for~} \alpha \not\in \ZZ,$$
and therefore  
$$\Gr_{\alpha }^V \big(\Mmod (*Y) \otimes L^{-i} \big) = 0 \,\,\,\,{\rm for~} \alpha + i / m \not\in \ZZ.$$
The bottom line is that in order to have $\Gr_{\alpha }^V \big(\Mmod (*Y) \otimes \tilde L \big) \neq 0$, one must have 
$\alpha + i/m \in \ZZ$ for all $1\le i \le m-1$, and consequently $\alpha$ cannot be an integer (recall that we are 
working with $m \ge 2$).

Let us now denote $\Mmod^\prime = \Mmod(*Y) \otimes \tilde L$ for simplicity.
Using this last remark, by Lemma \ref{extension_filtration} we deduce that $\Mmod^\prime$ is generated as a 
$\Dmod$-module by the negative part of its  $V$-filtration, i.e.
$$\Mmod^\prime \simeq V_{<0} \Mmod^\prime \cdot \Dmod_{\PP^N}.$$
The next thing to note is that, again since the jumps in the $V$-filtration do not happen at integers, according to 
Lemma \ref{equivalence}(2) the second condition in Definition \ref{quasi-unipotent} is equivalent to the fact that 
$$F_p \Mmod^\prime = \sum_{i\ge 0} \big( F_{p-i} V_{<0}\Mmod^\prime \big)\cdot \partial_t^i$$
for all $p$. Consequently, the Hodge filtration is determined by its restriction to the negative part of the $V$-filtration. 
Finally, this restriction is determined by the corresponding filtration on the open complement $U$ since according to 
Lemma \ref{equivalence}(1) for all $p$ we have
$$F_p V_{< 0} \Mmod^\prime = V_{<0} \Mmod^\prime \cap j_* j^{-1} F_p \Mmod^\prime.$$
As $(\tilde \Mmod, F)$ and $(\Mmod (*Y), F)\otimes \tilde L$
coincide on $U$, and as $(\tilde \Mmod, F)$ was defined by extension from $U$, the two filtered $\Dmod$-modules 
must then agree everywhere.
\end{proof}

\subsection{Particular cases}\label{particular_cases}
In this section I will explain how M. Saito's vanishing theorem can be used to deduce many of the standard 
vanishing theorems. In the next sections I will make the point however that the abstract version is equally 
valuable for concrete applications.

\noindent
{\bf Kodaira-Nakano vanishing.}
Let $X$ be a smooth projective complex variety of dimension $n$.
We consider the trivial Hodge module $M = \QQ_X^H \decal{n}$. 
According to Example \ref{basic}, the corresponding right $\Dmod$-module is $\omega_X$, with filtration
$F_p \omega_X = \omega_X$  if  $p \ge -n$ and $F_p \omega_X = 0$ if $ p < -n$, so that
$$\Gr_{-p}^F \DR_X(\omega_X) = \Omega_X^p \decal{n-p} \,\,\,\, {\rm for ~all~} p.$$ 
Theorem \ref{saito_vanishing} gives
$$H^q (X, \Omega_X^p \otimes L) = 0 \,\,\,\,{\rm for~} p+ q > n$$
and the dual statement, for any $L$ ample, i.e. Kodaira-Nakano vanishing.

If we restrict to the Kodaira vanishing theorem, which corresponds to the lowest non-zero piece of the filtration on 
$\omega_X$, then 
we can see it as an example of the following more easily stated special case of Theorem \ref{saito_vanishing}; it is useful to record this for applications.

\begin{corollary}\label{lowest}
If $(\Mmod, F)$ is a filtered $\Dmod$-module underlying a mixed Hodge module $M$ on a 
projective variety $X$, and $L$ is an ample line bundle on $X$, then
$$H^i (X, S(\Mmod) \otimes L ) = 0 \,\,\,\,{\rm for~all~} i > 0.$$
\end{corollary}

\noindent
{\bf Koll\'ar  vanishing.} The following theorem of Koll\'ar is a natural generalization of Kodaira vanishing to higher direct images of canonical bundles. 

\begin{theorem}[\cite{Kollar1}*{Theorem 2.1(iii)}]\label{kollar}
Let $f: X \rightarrow Y$ be a morphism between complex projective varieties, with $X$ smooth, and let $L$ be an ample 
line bundle on $Y$. Then
$$H^i (X, R^j f_* \omega_X \otimes L) = 0 \,\,\,\, {\rm for ~all~} i > 0 \,\,\,\,{\rm and~all~}j.$$
\end{theorem}

To deduce the statement from Theorem \ref{saito_vanishing}, 
we consider the push-forward $M = f_* \QQ_X^H \decal{n}$ of the 
trivial Hodge module on $X$, with $n = \dim X$. 
According to Example \ref{direct_image}, for the underlying $\Dmod$-modules we have  
$$f_+ (\omega_X, F)  \simeq \bigoplus_i (\Mmod_i, F) [-i]$$ 
in the derived category of filtered $\Dmod_Y$-modules (so compatible with inclusions into smooth varieties), and for 
each $i$ we have $S(\Mmod_i) = R^i f_* \omega_X$. Theorem \ref{kollar} then follows from Corollary \ref{lowest}.
More generally, the same argument shows the following vanishing theorem due to Saito: the statement of Theorem \ref{kollar} 
holds for $R^i f_* S(\Mmod)$, where $\Mmod$ corresponds to the unique pure Hodge module with strict support $X$ extending a polarized variation of Hodge structure on an open set $U \subseteq X$.

\noindent
{\bf Nadel vanishing.}
To deduce Nadel vanishing, one needs a more subtle relationship between 
multiplier ideals, the $V$-filtration on the structure sheaf, and the Hodge filtration on localizations, 
combining results of \cite{Budur_Saito} and \cite{Saito-HF}. As mentioned in Example \ref{localization}, this is one place 
where it is more convenient to have the initial discussion in terms of left $\Dmod$-modules.

Let $X$ be a smooth projective variety, and $D$ an effective Cartier divisor on $X$. Recall that $\shO_X (*D)$ is equipped with a natural Hodge filtration $F$, as the left $\Dmod$-module associated to the Hodge module $j_* \QQ^H_U \decal{n}$, where $j: U = X \smallsetminus D \hookrightarrow X$  is the inclusion; see Example \ref{localization}.
Looking at the first step in this filtration, one can recognize multiplier ideals from the formula
\begin{equation}\label{multiplier_formula}
F_0 \OX(\ast D) \simeq \mathcal{J} \bigl( (1-\eps) D \bigr) \cdot \OX(D),
\end{equation}
where $0 < \epsilon \ll 1$, and in general $\mathcal{J} (B)$ stands for the multiplier ideal of a $\QQ$-divisor $B$ (see 
\cite{Lazarsfeld}*{Ch.9}). Indeed, \cite{Saito-HF}*{Theorem~0.4} says that 
$$F_0 \OX(\ast D) \simeq {\tilde V}^1 \OX \cdot \OX(D),$$ 
while \cite{Budur_Saito}*{Theorem~0.1} says that for any $\alpha \in \QQ$ one has
$${\tilde V}^{\alpha} \OX  \simeq \mathcal{J} \bigl( (\alpha-\eps) D \bigr).$$
Here the $V$-filtration notation means the following: assume that $D$ is given locally by an equation $f$, and 
consider the graph embedding $i_f: X \rightarrow X \times \CC$. One can consider the $V$-filtration on the left $\Dmod$-module 
${i_f}_* \shO_X = \shO_X \otimes_{\CC} \CC[\partial_t]$ along $X_0 = X \times \{0\}$. The notation ${\tilde V}^\alpha \shO_X$ stands for the filtration induced on $\shO_X = \shO_X \otimes 1$.

This allows us to deduce the Nadel vanishing theorem (see e.g. \cite{Lazarsfeld}*{Theorem 9.4.8}), at least when $D$ is a Cartier divisor.

\begin{theorem}
With the notation above, if $L$ is a line bundle on $X$ such that $L - D$ is ample, then 
$$H^i \big(X, \omega_X \otimes L \otimes \mathcal{J} \bigl( (1-\eps) D \bigr)\big ) = 0 \,\,\,\, {\rm for~all~} i >0.$$
\end{theorem}
\begin{proof}
From the discussion above it follows that for the left $\Dmod$-module $\shO_X(*D)$ the lowest graded piece for the filtration $F$ is 
$$\Gr_0^F \shO_X(*D) = \mathcal{J} \bigl( (1-\eps) D \bigr) \cdot \OX(D),$$
so for the associated right $\Dmod$-module we have 
$$S (\Mmod) = \omega_X \otimes \shO_X (D) \otimes \mathcal{J} \bigl( (1-\eps) D \bigr).$$
Corollary \ref{lowest} implies that if $A$ is ample, then 
$$H^i (X, \omega_X \otimes A \otimes \shO_X (D) \otimes \mathcal{J} \bigl( (1-\eps) D \bigr) = 0 \,\,\,\, {\rm for~all~} i >0.$$
But by assumption we can write $L \simeq A \otimes \shO_X (D)$ with $A$ ample.
\end{proof}

\begin{remark}[Arbitrary $\QQ$-divisors]
The Nadel vanishing theorem for arbitrary $\QQ$-divisors $B$ is not in general a vanishing theorem for the lowest graded piece of the Hodge filtration corresponding to a mixed Hodge module; it is however a consequence of the same result. Roughly speaking one can reduce to the situation studied above after performing  a Kawamata covering construction to arrive at a Cartier divisor, using a bistrictness property of Hodge modules which allows us to deduce vanishing for the push-forward to the base, and finally passing to an eigensheaf of the push-forward. In other words multiplier ideals are naturally direct summands of Hodge theoretic objects, while
Theorem \ref{saito_vanishing} also applies to filtered direct summands of $\Dmod$-modules underlying mixed Hodge modules, again 
since the functor  $\Gr_k^F \circ \DR$ is exact. I thank N. Budur for this observation.

On the other hand, it is perhaps most natural to try and prove an analogue of the Kawamata-Viehweg vanishing theorem for $\QQ$-divisors in the context of mixed Hodge modules. This will be done in Theorem 
\ref{kawamata_viehweg_MHM} below. An analogous extension of Nadel vanishing is then an immediate consequence; see Corollary \ref{nadel_vanishing_MHM}.
\end{remark}


\noindent
{\bf Abelian varieties.}
In the case of abelian varieties it turns out that  Theorem \ref{saito_vanishing} holds directly for the graded pieces of a filtered $\Dmod$-module $(\Mmod, F)$ underlying a Hodge module itself, rather than those of its de Rham complex.

\begin{proposition}[\cite{PS}*{Lemma 2.5}]\label{abelian}
Let $A$ be a complex abelian variety, $(\Mmod, F)$ the filtered $\Dmod$-module underlying a mixed Hodge module on
$A$, and let $L$ be an ample line bundle. Then for each $k \in \ZZ$, we have 
\[
        H^i \bigl( A, \Gr_k^F \Mmod \tensor L \bigr) = 0 \,\,\,\, {\rm for~all ~} i >0.
\]
\end{proposition}
\begin{proof}
Denote $g = \dim A$. Consider for each $k \in \ZZ$ the complex of coherent sheaves
\[
	\Gr_k^F \DR_A(\Mmod) = \Big\lbrack
		\Gr_k^F \Mmod \to \OmA{1} \tensor \Gr_{k+1}^F \Mmod \to \dotsb \to
			\OmA{g} \tensor \Gr_{k+g}^F \Mmod
	\Big\rbrack,
\]
supported in degrees $-g, \dotsc, 0$. According to Theorem~\ref{saito_vanishing},
this complex has the property that, for $i > 0$, 
\[
	H^i \bigl( A, \Gr_k^F \DR_A(\Mmod) \tensor L \bigr) = 0.
\]
Using the fact that $\OmA{1} \simeq \OA^{\oplus g}$, one can deduce the asserted
vanishing theorem for the individual sheaves $\Gr_k^F \Mmod$ by induction on $k$. Indeed, since
$\Gr_k^F \Mmod = 0$ for $k \ll 0$, inductively one has for each $k$ a distinguished
triangle 
\[	
	E_k \to \Gr_k^F \DR_A(\Mmod) \to \Gr_{k+g}^F \Mmod \to E_k[1],
\]
with $E_k$ an object satisfying $H^i ( A, E_k \tensor L) = 0$.
\end{proof}

This observation is one of the key points towards showing that, under the above assumptions, all graded pieces $ \Gr_k^F \Mmod$ satisfy the analogues of the generic vanishing theorems of \cite{GL1}, \cite{GL2}, \cite{Hacon}, \cite{PP}. 
In view of the examples in Section \ref{examples}, besides recovering these results the statement leads to new 
applications, for instance to Nakano-type generic vanishing (see \cite{PS}*{Theorem 1.2}).

\begin{theorem}[\cite{PS}, Theorem 1.1]
Under the assumptions of Proposition \ref{abelian}, each $\Gr_k^F \Mmod$ is a $GV$-sheaf on $A$, i.e.
$${\rm codim}_{{\rm Pic}^0 (A)}~\{\alpha \in {\rm Pic}^0 (A)~|~ H^i (A, \Gr_k^F \Mmod \otimes \alpha) \neq 0\} 
 \ge i, \,\,\,\, {\rm for~all} \,\,\,\, i \ge 0.$$
\end{theorem}

A stronger generic vanishing statement  was proved in \cite{PS} for the total associated graded object 
$$\Gr^F_{\bullet} \Mmod = \bigoplus_k \Gr_k^F \Mmod,$$
seen as a coherent sheaf on $T^*A \simeq A \times H^0 (A, \Omega_A^1)$; this was useful in proving that all holomorphic $1$-forms on varieties of general type have zeros \cite{PS2}.

\subsection{Weak positivity}\label{weak_positivity}

This section contains a proof of an extension of Viehweg's weak positivity theorem for direct images of relative canonical sheaves, based on Theorem \ref{saito_vanishing} and found jointly with C. Schnell; see also \cite{Schnell2}. The general strategy follows Koll\'ar's approach to semipositivity via vanishing theorems in \cite{Kollar1}*{\S3}. The shortness of the proof is due to the fact that one can apply the machinery  of vanishing theorems to abstract Hodge modules.

\begin{definition}
A torsion-free coherent sheaf $\shF$ on a (quasi-)projective variety $X$ is \emph{weakly positive} on a non-empty open set $U \subseteq X$ if for every ample line bundle $A$ on $X$ and every $a\in \NN$, the sheaf ${\hat S}^{ab} \shF \otimes A^{\otimes b}$ is generated by global sections at each point of $U$ for $b$ sufficiently large. 
(Here ${\hat S}^p \shF$ denotes the reflexive hull of the symmetric power $S^p \shF$.)
\end{definition}

Before proving the main result, let's record a standard global generation consequence of Theorem \ref{saito_vanishing}.

\begin{corollary}\label{regularity}
Let $X$ be a smooth projective complex variety of dimension $n$, and $(\Mmod, F)$ a filtered $\Dmod$-module on $X$ 
underlying a mixed Hodge module $M$. Then for any ample and globally generated line bundle $L$ on $X$, the sheaf
$$ S (\Mmod) \otimes L^{\otimes (n+1)}$$
is globally generated.
\end{corollary}
\begin{proof}
Denoting $\shF =  S (\Mmod) \otimes L^{\otimes (n+1)}$, Corollary \ref{lowest} implies that 
$$H^i (X, \shF \otimes L^{\otimes -i}) = 0 \,\,\,\,{\rm for~all~} i >0.$$
The result is then an immediate consequence of the Castelnuovo-Mumford Lemma; see \cite{Lazarsfeld}*{Theorem 1.8.5}.
\end{proof}

We also need the following simplification of what is needed in order to check weak positivity under our hypotheses.

\begin{lemma}\label{reduction}
Let $\shF$ be a torsion-free sheaf on a smooth (quasi-)projective variety $X$, and $L$ a line bundle on $X$. Then $\shF$ is weakly positive on an open set $U \subseteq X$ on which $\shF$ is locally free if  
$\shF^{\otimes a} \otimes L$ is generated by global sections over $U$ for all $a>0$.
\end{lemma}
\begin{proof}
This is well known, so I will only sketch the proof. First, it is standard that one can reduce to checking the definition for only one (not necessarily ample) line bundle $L$, and all $a >0$; see \cite{Viehweg1}*{Remark 1.3(ii)}. Now a  torsion-free sheaf is locally free and therefore coincides with its reflexive hull outside of a closed set of codimension at least $2$. 
On the other hand, its  global sections inject into those of the reflexive hull. So it is enough to reduce the definition to the usual symmetric powers, which in turn are quotients of the tensor powers. 
\end{proof}

Viehweg's theorem in \cite{Viehweg1} saying that $f_*\omega_{Z/X}$ is weakly positive for any 
surjective morphism 
$f: Z \rightarrow X$ of smooth projective varieties is a special case of the following result.\footnote{We apply
it to the left $\Dmod$-modules $\Nmod_i$ corresponding to $\Mmod_i$ in the decomposition 
$f_* (\omega_Z, F)  \simeq \bigoplus_i (\Mmod_i, F) [-i]$; see Example \ref{direct_image}.}
The statement and proof are more conveniently phrased in terms of left $\Dmod$-modules.

\begin{theorem}\label{viehweg}
Let $X$ be a smooth projective complex variety, and $(\Nmod, F)$ the filtered \emph{left} $\Dmod$-module on $X$ underlying a 
mixed Hodge module $M$  which is a variation of mixed Hodge structure on a non-empty open set $U\subseteq X$. Then $S(\Nmod)$ is weakly positive over $U$.
\end{theorem}
\begin{proof}
\emph{Step 1.} 
First, as $M$ is a variation of mixed Hodge structure generically over $X$, it is well known that $S(\Nmod)$ is a torsion-free sheaf on $X$.
Fix now a positive integer $m$, and consider the diagonal embedding 
$$i : X \hookrightarrow X \times \cdots \times X,$$
where the product is taken $m$ times. On this product, consider the box product mixed Hodge module 
$$M^{\boxtimes m} : = M\boxtimes \cdots \boxtimes M.$$
As the filtration on $M^{\boxtimes m}$ is the convolution of the filtrations on the individual factors, it is not hard to see that 
$p (\Nmod^{\boxtimes m}) = m \cdot p(\Nmod)$ and moreover
$$i^* S(\Nmod^{\boxtimes m}) = S(\Nmod)^{\otimes m}.$$
Denoting by $r = (m-1) n$ the codimension of $X$ via the diagonal embedding, in the derived category of coherent sheaves on $X$ 
we have a natural morphism
\begin{equation}\label{derived_restriction}
F_{m \cdot p(\Nmod) - r} ~ i^{!}  (\Nmod^{\boxtimes m}, F) \longrightarrow {\bf L} i^* S(\Nmod^{\boxtimes m}) [-r],
\end{equation}
which is an isomorphism over the open set $U$ where $M$ is a variation of mixed Hodge structure. 
This follows for instance from \cite{Schnell1}*{Lemma 2.17} (see also \cite{Schnell2}*{Lemma 3.2}).

\noindent
\emph{Step 2.} We can specialize formula (\ref{derived_restriction}) by passing to the cohomology sheaves in degree $r$, in order to obtain a natural sheaf homomorphism
\begin{equation}\label{restriction}
S(\mathcal{Q})= F_{m \cdot p(\Nmod) - r} ~\mathcal{Q} \longrightarrow  S(\Nmod)^{\otimes m}
\end{equation} 
which is an isomorphism on $U$; here $(\mathcal{Q}, F)$ is another filtered left $\Dmod$-module on $X$, underlying
the object $i^* M^{\boxtimes m}$ in MHM($X$).

Fix now a very ample line bundle $L$ on $X$. In order to deduce that $S(\Nmod)$ is weakly positive over $U$, using Lemma \ref{reduction} 
it suffices then to show that $S(\mathcal{Q}) \otimes \omega_X \otimes L^{\otimes(n+1)}$ is globally generated, where $n = \dim X$. But this a consequence of Corollary \ref{regularity}, recalling that $S(\mathcal{Q}) \otimes \omega_X$ is the lowest non-zero graded piece of the right 
$\Dmod$-module associated to $\mathcal{Q}$.
\end{proof}

\begin{remark}
A more general result, involving kernels of Kodaira-Spencer morphisms associated to the de Rham complex of $\Mmod$, was recently 
proved in \cite{PW}. The method of proof is however different, and does not rely on vanishing theorems.
\end{remark}

In \cite{Viehweg1}, Viehweg proved that if $f: Z \rightarrow X$ is a surjective morphism of smooth projective varieties, then $f_*\omega_{Z/X}^{\otimes m}$ is weakly positive for $m \ge 2$ as well. A natural question to ask in this direction is the following:

\begin{question}
Let $f:Z \rightarrow X$ be a surjective morphism of smooth projective varieties, 
and $(\Mmod, F)$ the filtered left $\Dmod$-module underlying a 
mixed Hodge module $M$ which is a variation of mixed Hodge structure on a non-empty open set 
in $Z$. Is
$$f_* \left(S(\Nmod) \otimes \omega_{Z/X}^{\otimes m} \right)$$
weakly positive for all $m\ge 1$?
\end{question}

Assuming a positive answer to this question, the exact same method of proof as in Theorem \ref{viehweg} would 
imply for all $m \ge 2$ the weak positivity of 
$$f_* \left(S(\Nmod)^{\otimes m}  \otimes \omega_{Z/X}^{\otimes m} \right).$$
It is worth noting that it is indeed now possible to give a proof of Viehweg's statement on $f_*\omega_{Z/X}^{\otimes m}$
using cohomological methods \`a la Koll\'ar; see \cite{PS3}.

\subsection{Kawamata-Viehweg-type vanishing}\label{kawamata_viehweg}
In this section I will show that the Kawamata-Viehweg vanishing theorem for $\QQ$-divisors continues to hold for the lowest graded piece of a mixed Hodge module as long as its singular locus does not intersect the augmented base locus 
${\bf B}_{+} (L)$ of a big and nef line bundle (in particular always for variations of mixed Hodge structure). 
The proof follows quite closely the original one, with modifications permitted by Saito's study of non-characteristic pullbacks. I expect a stronger version to hold, at least under certain non-characteristicity hypotheses with respect to ${\bf B}_{+} (L)$.\footnote{Added during revision: since this was written, in the case when $L$ is a big and nef line bundle the most general version of Kawamata-Viehweg-type vanishing was proved by Suh \cite{Suh} and Wu \cite{Wu}. Further results for $\QQ$-divisors were also obtained in \cite{Wu}.}

\begin{theorem}\label{kawamata_viehweg_MHM}
Let $(\Mmod, F)$ be the filtered right $\Dmod$-module underlying a mixed Hodge module $M$ on a smooth projective variety $X$, and let $L$ be a line bundle on $X$ with $L\sim_{\QQ} A + \Delta$, where $A$ is a big and nef $\QQ$-divisor on $X$ and $(X, \Delta)$ is a klt pair.  Assume that $\B+ (A)\cup {\rm Supp} ~\Delta$ is contained in the smooth locus of $M$. Then
$$H^i (X, S(\Mmod) \otimes L) = 0 \,\,\,\, {\rm for~all~} i >0.$$
\end{theorem}

\begin{remark}
In particular we have the vanishing above if $L$ is a big and nef line bundle such that $\B+ (L)$ is contained in the 
smooth locus of $M$.
Note that one does not have a similar statement for other associated graded pieces $\Gr_k^F \DR (\Mmod)$ of the filtered de Rham complex, as in the case of Kodaira-Saito vanishing. This is already well known for the trivial Hodge module  
$M = \QQ_X^H[n]$. In this case, by Example \ref{basic} the graded pieces are $\Omega^k_X [n-k]$ with
$n = \dim X$. Simple examples show however that for $k < n$ the Nakano extension of Kodaira vanishing does not usually hold for twists by big and nef line bundles; see \cite{Lazarsfeld}*{Example 4.3.4}.
\end{remark}

In order to understand the statement and proof, we need to review a few more definitions and results. Before doing this, let's note that an immediate consequence of the theorem above is the following generalization of the Nadel vanishing theorem; see also Section \ref{particular_cases}.

\begin{corollary}\label{nadel_vanishing_MHM}
Let $X$ be a smooth projective variety, and $D$ an effective $\QQ$-divisor on $X$ with associated multiplier ideal $\I (D)$. Let $L$ be a line bundle in $X$ such that $L - D$ is big and nef, and assume that 
$\B+( L - D) \cup {\rm Supp}~D$ is contained in the smooth locus of a mixed Hodge module $M$ with underlying filtered $\Dmod$-module $(\Mmod, F)$. Then
$$H^i \big(X, S(\Mmod) \otimes L \otimes \I (D) \big) = 0 \,\,\,\,{\rm for~all~} i >0.$$ 
\end{corollary}

\noindent
{\bf Higher direct images of the lowest Hodge piece.}
Let $X$ be a smooth variety.  Recall that according to M. Saito's theory \cite{Saito-MHM}, for a mixed Hodge module $M$ with strict support equal to $X$, there exists a maximal non-empty open set $U \subseteq X$ on which $M$ is variation of mixed Hodge structure, denoted say by $\V$; we call this the \emph{smooth locus} of $M$. Note that the lowest
Hodge piece $S(\Mmod)$ is a locally free sheaf on $U$.

As the functor $S(\cdot)$ is exact, we can often restrict our study to the case when $M$ is a pure Hodge module
which is a polarized variation of Hodge structure on $U$. In this case, in response to a
 conjecture of Koll\'ar, Saito proved (among other things) the following, the second part of which 
 can be seen as a generalization of the Grauert-Riemenschneider vanishing theorem. 
 
 \begin{theorem}[Saito, \cite{Saito-Kollar}]\label{saito_GR}
 Let $f: X \rightarrow Y$ be  a surjective projective morphism (with $Y$ possibly singular), and let $(\Mmod, F)$ be the 
filtered $\Dmod$-module underlying a pure Hodge module with strict support $X$ that is generically a polarized variation of Hodge structure $\V$. For each $i\ge 0$, one has 
 $$R^i f_* S(\Mmod) = S(Y, \V^i),$$
 the lowest Hodge piece of the variation of Hodge structure $\V^i$ on the intersection cohomology of $\V$ along
 the fibers of $f$. Consequently, $R^i f_* S(\Mmod)$ are torsion-free, and in particular 
 $$R^i f_* S(\Mmod) = 0 \,\,\,\,{\rm for~} i > \dim X - \dim Y.$$
 \end{theorem}

\noindent
{\bf Augmented base loci.}
We start by recalling the definition and some basic results on augmented base loci of divisors.

\begin{definition}[\cite{ELMNP}*{\S1}]\label{augmented}
Let $D$ be a $\QQ$-divisor on a normal complex projective variety $X$. The \emph{augmented base locus} of $L$ is 
$$\B+ (D) : = \BB (D - \epsilon H),$$
where $H$ is any ample divisor on $X$, $0 < \epsilon \ll 1$ is rational, and $\BB (D- \epsilon H)$ denotes the stable base locus of the $\QQ$-divisor $D - \epsilon H$, i.e. the base locus of $|m (D - \epsilon H)|$ for $m \gg 0$. 
If $L$ is a line bundle, we define $\B+ (L)$ similarly. It is not hard to check (see \cite{ELMNP}*{Proposition 1.5}) that
equivalently one has 
\begin{equation}\label{kodaira_decomp}
\B+ (D) = \bigcap_{D = A+ E} {\rm Supp}~ E,
\end{equation}
where the intersection is taken over all $\QQ$-divisor  decompositions of $D$ such that $A$ is ample and $E$ is effective.

We have that $\B+ (L) \neq X$ if and only if $L$ is big. When $L$ is a big and nef, according to Nakamaye's theorem \cite{Nakamaye}, one has the following description 
$$\B+ (L) = {\rm Null} (V),$$ where ${\rm Null}(V)$ is the union of all subvarieties $V\subset X$ such that $L^{\dim V} \cdot V = 0$, or equivalently $L_{|V}$ is not big.
\end{definition}

We will use the following birational interpretations of the augmented base locus; slightly more general statements can be found for instance in \cite{BBP}*{Lemma  2.2 and Proposition 2.3}.\footnote{I thank Angelo Lopez for pointing 
out this reference.}

\begin{lemma}\label{kodaira_birational}
If $D$ is a  $\QQ$-divisor on $X$, then
$$\B+ (D)=\bigcap_{f,A,E}f \left({\rm Supp} ~E \right),$$
where the intersection is taken over all projective birational morphisms
$f: Y \rightarrow X$ with $Y$ normal, and all decompositions
$f^*D \sim_{\QQ} A+E$, with $A$ ample and $E$ effective.
\end{lemma}

\begin{lemma}\label{birational_augmented}
Let $f: Y\rightarrow X$ be a birational morphism of smooth projective varieties, and ${\rm Exc}(f)\subseteq Y$ 
its exceptional locus. If $D$ is a $\QQ$-divisor on $X$, then
$$\B+ (f^*(D))=f^{-1}(\B+(D))\cup {\rm Exc}(f).$$
\end{lemma}

\noindent
{\bf Proof of Theorem \ref{kawamata_viehweg_MHM}.}
First, just as in the proof of Saito's vanishing theorem, due to the exactness of the functor $S(\cdot)$ we can reduce to assuming that $M$ is a pure Hodge module.
I will divide the proof into a few steps which loosely follow the standard steps in the proof of the Kawamata-Viehweg theorem.  In the first three steps we will assume that $L$ is a big and nef line bundle, and $\Delta =  0$. The last two will deal with the general case.

\noindent
\emph{The line bundle case.}
Note to begin with that since 
$L$ is big, in general there exist an $m>0$, an ample line bundle $A$, and an effective divisor $E$, such that 
\begin{equation}\label{kodaira}
L^{\otimes m} \simeq A\otimes \shO_X(E).
\end{equation}

\noindent
\emph{Step 1.}
This is a Norimatsu-type statement (see \cite{Lazarsfeld}*{Lemma 4.3.5}): we show that 
if $A$ is an ample line bundle, and 
$E\subset X$ is a reduced simple normal crossings divisor on $X$ contained in the smooth locus of $M$, then 
$$H^i \big(X, S(\Mmod) \otimes A \otimes \shO_X(E)\big) = 0 \,\,\,\,{\rm for~all~} i >0.$$

Let's assume first that $E$ is a smooth divisor. As $S(\Mmod)$ is locally free in a neighborhood of $E$, we have a short exact sequence
$$0 \longrightarrow S(\Mmod) \otimes A \longrightarrow S(\Mmod) \otimes A \otimes \shO_X(E) 
\longrightarrow S(\Mmod)_{|E} \otimes A_{|E} \otimes \shO_E (E) \longrightarrow 0
$$
Passing to cohomology and applying Corollary \ref{lowest}, we see that is is enough to show that 
$$H^i \big(E, S(\Mmod)_{|E} \otimes \shO_E (E)  \otimes A_{|E}  \big) = 0 \,\,\,\,{\rm for~all~} i >0.$$
Again by Corollary \ref{lowest}, it suffices then to note that $S(\Mmod)_{|E} \otimes \shO_E (E) \simeq 
S(\Mmod^\prime)$, for some filtered $\Dmod$-module underlying a mixed Hodge module $M^\prime$ on $E$. We can in fact take
$$(\Mmod^\prime, F) : = (\mathcal{H}^1 i^{!}  \Mmod, F).$$ 
On one hand, this filtered  $\Dmod$-module underlies a Hodge module, as 
$$\mathcal{H}^1 i^{!}  \Mmod, F) \simeq i^{!} (\Mmod, F) [1]$$ 
by \cite{Saito-MHP}*{Lemme 3.5.6}. On the other hand, since $E$ is contained in the smooth locus of $M$, using 
\cite{Schnell1}*{Lemma 2.17} (as in the proof of Theorem \ref{viehweg}) we see that there is an isomorphism 
$S(\Nmod^\prime) \simeq S(\Nmod)_{|E}$, 
where $\Nmod$ is again notation for the associated left $\Dmod$-modules. This is equivalent to what we want by adjunction.

In general we have $E= E_1 + \cdots + E_k$, where $E_j$ are smooth divisors with transverse intersections. 
The statement can be easily proved by induction on $k$, using exact sequences similar to the one above, and the fact that $M$ continues to be a variation of mixed Hodge structure when restricted to the log-canonical centers of 
$E$.

\noindent
\emph{Step 2.}
In this step we show that we can reduce the general statement to the case where in $(\ref{kodaira})$ we have that $E$ has simple normal crossings support, and this 
support is contained in the smooth locus of $M$.
Consider the notation of Definition \ref{augmented}, so that 
$$\B+ (L) = \BB (L - \epsilon H) = {\rm Bs} \left( L^{\otimes k} \otimes \shO_X (- k \epsilon H) \right),$$
for $k$ sufficiently large and divisible, and ${\rm Bs}(\cdot)$ stands for the usual base locus. We consider 
$\mu:Y \rightarrow X$ a log-resolution of the linear series $|L^{\otimes k} \otimes \shO_X (- k \epsilon H)|$, so that 
$$\mu^* \left( L^{\otimes k} \otimes \shO_X (- k \epsilon H) \right) \simeq M_k \otimes \shO_Y (F_k),$$
where $M_k$ is the moving part of the pullback, a big and basepoint-free line bundle, and $F_k$ is its fixed divisor. 
From Lemma \ref{birational_augmented} we have that 
$$\B+ (\mu^* L) = \mu^{-1} (\B+ (L)) \cup {\rm Exc(\mu)} = {\rm Supp} ( F_k) \cup~ {\rm Exc(\mu)},$$ 
which is a divisor with simple normal crossings support on Y.

By assumption $\B+ (L)$ is contained in the smooth locus of $M$. Choosing the log-resolution to be an isomorphism outside of
 $\B+(L)$, by Example \ref{nonchar_example} we have that $\mu$ is non-characteristic for $(\Mmod, F)$. Recall that this implies that the filtered inverse image 
$\mu^*(\Mmod, F) = (\tilde \Mmod, F)$ is given by the formula 
$$\tilde \Mmod = \mu^{-1} \Mmod \otimes_{\mu^{-1} \shO_X} \omega_{Y/X} \,\,\,\,{\rm and} \,\,\,\, 
F_p \tilde \Mmod = \mu^* F_p \Mmod \otimes \omega_{Y/X} ,$$
and this underlies the Hodge module $\mu^*M$. 
We see then that $S(\mu^* \Mmod) \simeq \mu^* S(\Mmod) \otimes 
\omega_{Y/X}$, and so 
$$\mu_* S(\mu^* \Mmod) \simeq S(\Mmod),$$
as $\mu_* \omega_{Y/X} \simeq \shO_X$. Assuming that we proved that 
\begin{equation}\label{intermediate}
H^i (Y, S(\mu^*\Mmod) \otimes \mu^*L) = 0 \,\,\,\, {\rm for~all~} i >0,
\end{equation}
this implies the vanishing we want on $X$ as $R^i \mu_* S(\mu^*M) = 0$, which is a consequence of 
Theorem \ref{saito_GR}.

Let's now write 
$$F_k = \sum_j a_j E_j,$$
with the convention that $a_j \ge 0$, so that we may assume that the sum contains all the exceptional divisors of $\mu$ among the $E_j$. 
By construction we have that $\B+ (\mu^* L) $ is contained in the smooth 
locus of $\mu^* M$; equivalently, this statement holds for all $E_j$ in the sum above.

Finally, note that by construction we have
$$\mu^* L^{\otimes k} \simeq  \mu^* \shO_X (k \epsilon H) \otimes  M_k  \otimes \shO_Y (F_k),$$
and the line bundle $\mu^* \shO_X (k \epsilon H) \otimes  M_k $ is still big and nef.
To conclude, one appeals to a version of the Negativity Lemma, stating that for such a $k \gg 0$, there exist $b_j \ge 0$ such that 
$$\mu^* \shO_X (k\epsilon H) \otimes \shO_Y ( - \sum_j b_j F_j)$$ 
is ample, where the sum runs over the exceptional divisors of $\mu$ 
(and so with the same convention as above we can assume that it runs over all $E_j$);  
see e.g. \cite{Lazarsfeld}*{Corollary 4.1.4}. But now we can write 
$$\mu^* L^{\otimes k} \simeq \big( \mu^* \shO_X (k\epsilon H) \otimes \shO_Y ( - \sum_j b_j F_j) \big) \otimes 
\shO_Y \big(\sum_j (a_j + b_j)F_j \big),$$
which is of the form required at the beginning of this reduction step.

\noindent
\emph{Step 3.} In this last step we conclude the proof assuming that $E$ in $(\ref{kodaira})$ has simple normal crossings support contained in the smooth locus of $M$, which is the outcome of Step 2.
By standard arguments using Kawamata covers, it is known that there exists a finite cover $f: Y \rightarrow X$
with $Y$ smooth projective, such that 
$$f^* L \simeq A^\prime \otimes  \shO_Y (E^\prime),$$
with $A^\prime$ ample and $E^\prime$ a reduced simple normal crossings divisor; see e.g.\cite{Lazarsfeld}*{p.255}. 
Moreover, $f$ can be chosen to be non-characteristic with respect to $(\Mmod, F)$. 

This last statement requires some discussion; recall that Kawamata covers can be constructed in two steps (see \cite{Lazarsfeld}*{Proposition 4.1.12}). The first is a Bloch-Gieseker type cover $g: Z \rightarrow X$, where for some component $E_1$ of $E$ one can write $g^*E = kE_1$, for a given $k$ and some $E_1$ not necessarily effective. In this step one can assume that 
$E$ is very ample by writing it as the difference of two very ample line bundles, and then $g$ can be constructed so as to be ramified along a generic union of hyperplane sections of $X$ in the embedding given by $E$; see the proof of \cite{Lazarsfeld}*{Theorem 4.1.10}. From this genericity it follows that $g$ is non-characteristic with respect to 
$(\Mmod, F)$. On the other hand, the second step is to consider a cyclic cover $h: Y \rightarrow Z$, which is ramified 
along $f^*E_1$; since this is contained in the smooth locus of $f^*M$, this cover is also non-characteristic.
One then applies this procedure inductively for all components of $E$.

Going back to the proof, we can now consider the filtered inverse image $f^*(\Mmod, F)$ underlying the pullback Hodge module just as in Step 2. Note that we have $E^\prime = f^{-1} ({\rm Supp}~E)$, and so $E^\prime$ is contained 
in the smooth locus of $f^*M$. By Step 1, we then have 
$$H^i (Y, S(f^*\Mmod) \otimes f^*L) = 0 \,\,\,\, {\rm for~all~} i >0.$$
But precisely as in Step 2 we have that 
$$f_* S(f^* \Mmod) \simeq S(\Mmod) \otimes f_* \omega_{Y/X}.$$
As $\shO_X$ is a direct summand of $f_* \omega_{Y/X}$ via the trace map, we obtained the desired vanishing using the projection formula.

\noindent
\emph{The $\QQ$-divisor case.}
We do this in two further steps which reduce us to the line bundle case discussed above. We  first reduce to the case when ${\rm Supp}~\Delta$ is a simple normal crossings divisor.

\noindent
\emph{Step 4.}
Let $\mu: Y \rightarrow X$ be a log-resolution of $(X, \Delta)$, and write
$$K_Y - \mu^* (K_X + \Delta) = P - N,$$
where $P$ and $N$ are effective $\QQ$-divisors with simple normal crossings support, without common components, 
and such that $P$ is exceptional and all the coefficients in $N$ are strictly less than $1$. We then have 
$$K_Y + N + \lceil P\rceil - P + \mu^* A = \mu^* (K_X + \Delta + A) +  \lceil P\rceil,$$
and so there exists a line bundle $L^\prime$ on $Y$ such that $L^\prime \sim_\QQ \mu^*A + \Delta^\prime$, where 
$\Delta^\prime = N + \lceil P\rceil - P$, a strictly boundary divisor with normal crossings support. Note that $\mu^*A$
is still big and nef, and in fact by Lemma \ref{birational_augmented} we have
$$\B+ (\mu^* A) = \mu^{-1} (\B+ (A)) \cup {\rm Exc(\mu)}.$$
We can choose $\mu$ such that it is an isomorphism outside the support of $\Delta$. It follows that 
both $\B+ (\mu^* A)$ and ${\rm Supp}~\Delta^\prime$ are contained in the smooth locus of $\mu^*M$. Note finally 
that it is enough to show that
$$H^i (Y, S(\mu^*\Mmod) \otimes L^\prime) = 0 \,\,\,\,{\rm for~all~} i>0.$$
Indeed, we have observed before that 
$$\mu_* S(\mu^* \Mmod) \simeq S(\Mmod) \,\,\,\,{\rm and}\,\,\,\, R^i \mu_* S(\mu^* \Mmod) = 0 \,\,\,\,{\rm for~} 
i >0.$$

\noindent
\emph{Step 5.}
It is enough to assume then that $\Delta$ is a divisor with simple normal crossings, support, say 
$\Delta = \sum_{i=1}^{k} a_i D_i$, with $0 < a_i < 1$ and $D_i$ smooth. 

The strategy is to prove the statement 
by induction on $k$. The case $k = 0$ is the line bundle case proved above.
Assume now that $k >0$, and let's write $a_1 = \frac{p}{q}$. Note that $0 < p \le q-1$.
Just as in Step 3, one considers a Kawamata cover associated to the divisor $D_1$; concretely, 
there exists a finite morphism $f: Y \rightarrow X$, with $Y$ smooth projective, such that on $Y$ the divisor $D_1$ 
becomes divisible by $d$. In other words, we have
$$L^\prime : = f^*L \sim_{\QQ} A^\prime + cD_1^\prime + \sum_{i=2}^{k} a_i D_i^\prime,$$
where $A^\prime = f^* A$ and $D_i^\prime = f^* D_i$, still satisfying the fact that $\sum D_i^\prime$ has simple 
normal crossings. Moreover, this morphism can be chosen to be non-characteristic for $(\Mmod, F)$, so we can deal with $f^*M$ as in the previous proof. 

By induction we can now assume that the line bundle $L^\prime \otimes \shO_Y (- c D_1^\prime)$ satisfies
$$H^i \left(Y, S(f^* \Mmod)\otimes L^\prime \otimes \shO_Y (- c D_1^\prime)\right) = 0 \,\,\,\, {\rm for~all~} i >0.$$
Recall that due to the definition of the filtration under non-characteristic inverse image we have 
$S(f^* \Mmod) \simeq f^*S(\Mmod) \otimes \omega_{Y/X}$. On the other hand, it is standard that in the covering construction 
above we have that $f_* \left( L^\prime \otimes \shO_Y (- c D_1^\prime)  \otimes \omega_{Y/X} \right)$ contains $L$ as a direct 
summand. The desired vanishing follows from the projection formula.


\section*{References}

\begin{biblist}
\bib{BBP}{article}{ 
    author={Boucksom, S\'ebastien},
    author={Broustet, Ama\"el},
    author={Pacienza, Gianluca},
    title={Uniruledness of stable base loci of adjoint linear systems via Mori theory},  
    journal={Math. Z.},
    volume={275},
    date={2013}, 
    pages={499--507},
}
\bib{Budur_Saito}{article}{
   author={Budur, Nero},
   author={Saito, Morihiko},
   title={Multiplier ideals, $V$-filtration, and spectrum},
   journal={J. Algebraic Geom.},
   volume={14},
   date={2005},
   number={2},
   pages={269--282},
}
\bib{ELMNP}{article}{ 
    author={Ein, Lawrence},
    author={Lazarsfeld, Robert},
    author={Musta\c{t}\v{a}, Mircea},
    author={Nakamaye, Michael},
    author={Popa, Mihnea},
    title={Asymptotic invariants of base loci},  
    journal={Ann. Inst. Fourier},
    volume={56},
    date={2006}, 
    pages={1701--1734},
}
\bib{EV}{book}{
   author={Esnault, H{\'e}l{\`e}ne},
   author={Viehweg, Eckart},
   title={Lectures on vanishing theorems},
   series={DMV Seminar},
   volume={20},
   publisher={Birkh\"auser Verlag},
   place={Basel},
   date={1992},
   pages={vi+164},
}
\bib{GL1}{article}{
   author={Green, Mark},
   author={Lazarsfeld, Robert},
   title={Deformation theory, generic vanishing theorems, and some
   conjectures of Enriques, Catanese and Beauville},
   journal={Invent. Math.},
   volume={90},
   date={1987},
   number={2},
   pages={389--407},
}
\bib{GL2}{article}{
   author={Green, Mark},
   author={Lazarsfeld, Ro{-}bert},
   title={Higher obstructions to deforming cohomology groups of line bundles},
   journal={J. Amer. Math. Soc.},
   volume={1},
   date={1991},
   number={4},
   pages={87--103},
}
\bib{Hacon}{article}{
	author={Hacon, Christopher},
	title={A derived category approach to generic vanishing},
	journal={J. Reine Angew. Math.},
	volume={575},
	date={2004},	
	pages={173--187},
}
\bib{HTT}{book}{
   author={Hotta, R.},
   author={Takeuchi, K.},
   author={Tanisaki, T.},
   title={D-modules, perverse sheaves, and representation theory},
   publisher={Birkh\"auser, Boston},
   date={2008},
}
\bib{Kollar1}{article}{
   author={Koll\'ar, J\'anos},
   title={Higher direct images of dualizing sheaves I},
   journal={Ann. of Math.},
   volume={123},
   date={1986},
   number={1},
   pages={11--42},
}
\bib{Kollar2}{article}{
   author={Koll\'ar, J\'a{-}nos},
   title={Higher direct images of dualizing sheaves II},
   journal={Ann. of Math.},
   number={124},
   date={1986},
   pages={171--202},
}
\bib{Lazarsfeld}{book}{
         author={Lazarsfeld, Robert},
         title={Positivity in algebraic geometry I $\&$ II},
         series={Ergebnisse der Mathematik und ihrer Grenzgebiete}, 
         number={48 $\&$ 49}, 
         publisher={Springer-Verlag, Berlin}, 
         date={2004},
}
\bib{Maisonobe_Sabbah}{article}{
    author={Maisonobe, Phillipe},
    author={Sabbah, Claude},
    title={Aspects of the theory of $\Dmod$-modules (Kaiserslautern 2002)},
	journal={at http://www.math.polytechnique.fr/cmat/sabbah/livres.html},
}
\bib{MP}{article}{
      author={Musta\c t\u a, Mircea},
      author={Popa, Mihnea},
	title={Hodge ideals},
	journal={in preparation}, 
	date={2016}, 
}
\bib{Nakamaye}{article}{
	author={Nakamaye, Michael},
     title={Stable base loci of linear series},
	journal={Math. Ann.},
	volume={318},
	date={2000},	
	pages={837--847},
}
\bib{PP}{article}{
	author={Pareschi, Giuseppe},
	author={Popa, Mihnea},
	title={GV-sheaves, Fourier-Mukai transform, and generic vanishing},
	journal={Amer. J. Math.},
	volume={133},
	date={2011},	
	pages={235--271},
}
\bib{PS}{article}{
	author={Popa, Mihnea},
	author={Schnell, Christian},
	title={Generic vanishing theory via mixed Hodge modules},
	journal={Forum of Math., Sigma},
	volume={1},
	date={2013},	
	pages={1--60},
}
\bib{PS2}{article}{
	author={Popa, Mihn{-}ea},
	author={Schnell, Christian},
	title={Kodaira dimension and zeros of holomorphic one-forms},
	journal={Ann. of Math.},
	volume={179},
	date={2014},
	pages={1--12},
}
\bib{PS3}{article}{
	author={Popa, Mihnea},
	author={Schnell, Christian},
	title={On direct images of pluricanonical bundles},
	journal={Algebra and Number Theory},
	volume={8},
      date={2014},
      pages={2273--2295},
}
\bib{PW}{article}{
      author={Popa, Mihnea},
      author={Wu, Lei},
	title={Weak positivity for Hodge modules},
	journal={to appear in Math. Res. Lett., preprint arXiv:1511.00290}, 
	date={2015}, 
}
\bib{Saito-MHP}{article}{
   author={Saito, Morihiko},
   title={Modules de Hodge polarisables},
   journal={Publ. RIMS}, 
   volume={24},
   date={1988}, 
   pages={849--995},
}
\bib{Saito-MHM}{article}{
   author={Saito, Mori{-}hiko},
   title={Mixed Hodge modules},
   journal={Publ. Res. Inst. Math. Sci.},
   volume={26},
   date={1990},
   number={2},
   pages={221--333},
}
\bib{Saito-Kollar}{article}{
   author={Saito, Morihiko},
   title={On Koll\'ar's conjecture},
   journal={Several complex variables and complex geometry, Part 2 (Santa Cruz, CA, 1989), Proc. Sympos. Pure Math.},
   volume={52},
   publisher={Amer. Math. Soc.},
   place={Providence, RI},
   date={1991},
   pages={509--517},
}
\bib{Saito-b}{article}{
   author={Saito, Mori{-}hiko},
   title={On $b$-function, spectrum and rational singularity},
   journal={Math. Ann.},
   volume={295},
   date={1993},
   number={1},
   pages={51--74},
}
\bib{Saito-HF}{article}{
   author={Saito, Morihiko},
   title={On the Hodge filtration of Hodge modules},
   language={English, with English and Russian summaries},
   journal={Mosc. Math. J.},
   volume={9},
   date={2009},
   number={1},
   pages={161--191, back matter},
}
\bib{Schnell1}{article}{
   author={Schnell, Christian},
   title={Complex analytic N\'eron models for arbitrary families of intermediate Jacobians},
   journal={Invent. Math.},
   volume={188},
   date={2012},
   number={1},
   pages={1--81},
}
\bib{Schnell2}{article}{
   author={Schnell, Ch{-}ristian},
   title={ Weak positivity via mixed Hodge modules},
   journal={dedicated to Herb Clemens, to appear, preprint arXiv:1401.5654},
   date={2013},
}
\bib{Schnell3}{article}{
   author={Schnell, Christian},
   title={On Saito's vanishing theorem},
   journal={to appear in Math. Res. Lett., preprint arXiv:1407.3295},
   date={2014},
}
\bib{Schnell-MHM}{article}{
	author={Schnell, Ch{-}ristian},
	title={An overview of Morihiko Saito's theory of mixed Hodge modules},
	journal={preprint arXiv:1405.3096},
	year={2014},
}	
\bib{Suh}{article}{
      author={Suh, Junecue},
	title={Vanishing theorems for mixed Hodge modules and applications},
	journal={preprint, to appear in J. Eur. Math. Soc.}, 
	date={2015}, 
}
\bib{Viehweg1}{article}{
        author={Viehweg, Eckart},
	title={Weak positivity and the additivity of the Kodaira dimension of certain fiber spaces},
	journal={Adv. Studies PureMath.}, 
	number={1},
	date={1983}, 
	pages={329--353},
}
\bib{Wu}{article}{
      author={Wu, Lei},
	title={Vanishing and injectivity theorems for Hodge modules},
	journal={to appear in Trans. Amer. Math. Soc., preprint arXiv:1505.00881}, 
	date={2015}, 
}
\end{biblist}
\end{document}